\documentclass[a4paper,reqno]{amsart}
\usepackage{amssymb, amsmath, amscd}
\usepackage[final]{graphicx}
\usepackage{wrapfig}
\usepackage{color}

\def\id{\operatorname{id}}
\newtheorem{theorem}{Theorem}[section]
\newtheorem{lemma}[theorem]{Lemma}
\newtheorem{remark}[theorem]{Remark}
\newtheorem{definition}[theorem]{Definition}

\makeatletter
 \@addtoreset{equation}{section}
\makeatother\begin{document}
\title{Moduli spaces of oriented Type~$\mathcal{A}$ manifolds of dimension at least 3}
\author{P. Gilkey}
\address{PG: Mathematics Department, University of Oregon,
 Eugene OR 97403, USA.}
\email{gilkey@uoregon.edu}
\author{J. H. Park}
\address{JHP:Department of
Mathematics, Sungkyunkwan University, Suwon, 16419 Korea.\newline
\phantom{............}School of Mathematics, Korea Institute for Advanced Study,\newline
\phantom{............}85 Hoegiro Dongdaemun-gu, Seoul, 02455 Korea.}
\email{parkj@skku.edu}
\begin{abstract}
We examine the moduli space of oriented locally homogeneous
manifolds of Type~$\mathcal{A}$ which have non-degenerate symmetric
Ricci tensor both in the setting of manifolds with torsion
and also in the torsion free setting where the dimension is at least 3. These exhibit
phenomena that is very different than in the case of surfaces. In dimension 3, we determine
all  the possible symmetry groups in the torsion free setting.

\end{abstract}
\maketitle
\section{Introduction}
Let $M$ be a smooth oriented manifold of dimension $m$;
if $M\subset\mathbb{R}^m$, then the orientation will be given by $dx^1\wedge\dots\wedge dx^m$. Let $\nabla$
be a connection on
the tangent bundle of $M$. One says that $\mathcal{M}:=(M,\nabla)$ is {\it torsion free} if $\nabla_\xi\eta-\nabla_\eta\xi =[\xi,\eta]$.
Let $\vec x:=(x^1,\dots,x^m)$ be a system of local coordinates on $M$. Adopt the {\it Einstein convention} and sum
over repeated indices to expand $\nabla_{\partial_{x^i}}\partial_{x^j}=\Gamma_{ij}{}^k\partial_{x^k}$ in terms of the
Christoffel symbols $\Gamma=(\Gamma_{ij}{}^k)$; the condition that $\nabla$ is torsion free is then equivalent
to the symmetry $\Gamma_{ij}{}^k=\Gamma_{ji}{}^k$. The importance of the torsion free condition lies in the following
result which permits one to normalize the coordinate system so that only
the second and higher order derivatives of the connection $1$-form play a role. As we shall not be using this result, we shall omit
the proof and refer the interested reader to, for example,  Lemma~3.5 of \cite{GPV15} for the proof.
\begin{theorem}
$\mathcal{M}$ is torsion free if and only if for every point $P$ of $M$, there exist coordinates centered at $P$ so that $\Gamma_{ij}{}^k(P)=0$.
\end{theorem}

Although much of Riemannian geometry involves the study of the Levi-Civita connection, which is without torsion,
in recent years connections which have torsion have played an important role in many developments. We refer,
for example, to work on B-metrics \cite{FI03,IM14,M12,S12}, on almost hypercomplex geometries \cite{M11}, on
string theory \cite{A06,FI02,GMW04,I04}, on spin geometries \cite{K10},  on torsion-gravity \cite{BM16,C12,D15,F13,F14,F15,LP13,LP13a,VCF15},
on contact geometries \cite{ACF05,FI03}, on almost product manifolds \cite{Me12}, non-integrable
geometries \cite{ACFH15,B15}, on the non-commutative residue
for manifolds with boundary \cite{WWY14},
on Hermitian and anti-Hermitian
geometry \cite{MG11}, CR geometry \cite{DL15}, hyper-K\"ahler with torsion supersymmetric sigma models \cite{FIS14,FS15,FS16,Sx12},
Einstein--Weyl gravity at the linearized level \cite{DEG13}, Yang-Mills flow with torsion \cite{GDLS14}, ESK theories \cite{RFSC13},
double field theory \cite{HZ13}, BRST theory \cite{FLM16},
and the symplectic and elliptic geometries of gravity \cite{CSE11}. Perhaps surprisingly, even the $2$-dimensional case is of interest;
connections on surfaces have been used to construct
new examples of pseudo-Riemannian metrics without a corresponding Riemannian counterpart \cite{CGGV09, CGV10, De, KoSe}.

\subsection{Local homogeneity}One says that $\mathcal{M}$ is {\it locally homogeneous} if given any two points $P$ and $Q$ of $M$, there exists
the germ of a diffeomorphism $\Phi$ taking $P$ to $Q$ which preserves $\nabla$.
Physically, the locally homogeneous setting is of particular interest as it corresponds to locally isotropic geometries.
One has the following examples of homogeneous
geometries:
\smallbreak\noindent{\bf Type~$\mathcal{A}$.} Let $M=\mathbb{R}^m$ and let
$\Gamma\in(\mathbb{R}^{2m})^*\otimes\mathbb{R}^m$ be constant. The
translation group $\mathbb{R}^m$ acts transitively on $M$
and preserves $\nabla$.
\smallbreak\noindent{\bf Type~$\mathcal{B}$.} Let $M=\mathbb{R}^+\times\mathbb{R}^{m-1}$ and let $\Gamma_{ij}{}^k=(x^1)^{-1}{C_{ij}{}^k}$
for $C\in(\mathbb{R}^{2m})^*\otimes\mathbb{R}^m$ constant. The $ax+b$ group $(x^1,x^2\dots)\rightarrow(ax^1,ax^2+b^2,\dots,ax^m+b^m)$
for $a>0$ and $\vec b=(b^2,\dots,b^m)\in\mathbb{R}^{m-1}$ acts transitively on $M$ and preserves $\nabla$.
\smallbreak\noindent{\bf Type~$\mathcal{C}$.} Let $\nabla$ be the Levi-Civita connection of a complete simply connected
pseudo-Riemannian manifold $M$ of constant sectional curvature.

\subsection{Two dimensional geometry}
The examples given above provide a complete family of models for the locally homogeneous surfaces; any locally
homogeneous surface admits a coordinate atlas modeled on one of these examples.
These classes are not disjoint. No surface is both Type~$\mathcal{A}$ and Type~$\mathcal{C}$.
However there are surfaces that are both Type~$\mathcal{A}$ and Type~$\mathcal{B}$ and there are surfaces which are
both Type~$\mathcal{B}$ and Type~$\mathcal{C}$. We refer to Opozda \cite{Op04} for a proof of the following result in the torsion
free setting and to Arias-Marco and Kowalski \cite{AMK08} for the extension to the case of surfaces with torsion.
We refer as well to \cite{Du, G-SG, KVOp2, KVOp, Opozda} for related work.

\begin{theorem}
Let $\mathcal{M}=(M,\nabla)$ be a locally homogeneous surface where $\nabla$
can have torsion. Then at least one of the following
three possibilities hold that describe the local geometry:
\begin{itemize}
\item[($\mathcal{A}$)] There exists a coordinate atlas so the Christoffel symbols
$\Gamma_{ij}{}^k$ are constant.
\item[($\mathcal{B}$)] There exists a coordinate atlas so the Christoffel symbols have the form
$\Gamma_{ij}{}^k=(x^1)^{-1}C_{ij}{}^k$ for $C_{ij}{}^k$ constant and $x^1>0$.
\item[($\mathcal{C}$)] $\nabla$ is the Levi-Civita connection of a metric of constant Gauss
curvature.
\end{itemize}\end{theorem}

\subsection{The Ricci tensor} The curvature operator $R$ and
 the Ricci tensor $\rho$ of an arbitrary connection are given by setting
$$
R(\xi,\eta):=\nabla_{\xi}\nabla_{\eta}-\nabla_{\eta}\nabla_{\xi}-\nabla_{[\xi,\eta]}
\text{ and }\rho(\xi,\eta):=\operatorname{Tr}\{\sigma\rightarrow R(\sigma,\xi)\eta\}\,.
$$
In terms of local coordinates,
\begin{equation}\label{E1.a}\begin{array}{l}
R_{ijk}{}^l=\partial_{x_i}\Gamma_{jk}{}^l-\partial_{x^j}\Gamma_{ik}{}^l
+\Gamma_{in}{}^l\Gamma_{jk}{}^n-\Gamma_{jn}{}^l\Gamma_{ik}{}^n,\\[0.05in]
\rho_{jk}=\partial_{x_i}\Gamma_{jk}{}^i-\partial_{x^j}\Gamma_{ik}{}^i+\Gamma_{in}{}^i\Gamma_{jk}{}^n-\Gamma_{jn}{}^i\Gamma_{ik}{}^n
\end{array}\end{equation}
Note that in this setting $\rho$ need not be symmetric.

\subsection{Type~$\mathcal{A}$ geometry}
Let $S^2(\mathbb{R}^m)$ denote the space of symmetric 2-cotensors on $\mathbb{R}^m$;
$\sigma=\sigma_{ij}dx^i\otimes dx^j\in S^2(\mathbb{R}^m)$ if and only if $\sigma_{ij}=\sigma_{ji}$.
The natural parameter spaces with which we shall be working are defined by:
$$
\mathcal{W}(m):=(\mathbb{R}^{2m})^*\otimes\mathbb{R}^m\text{ and }\mathcal{Z}(m):=S^2(\mathbb{R}^m)\otimes\mathbb{R}^m\,.
$$
Since $\mathcal{Z}(m)$ is a subset of $\mathcal{W}(m)$, properties true on $\mathcal{W}(m)$ are often inherited by
$\mathcal{Z}(m)$ and we shall thus often not mention $\mathcal{Z}(m)$ explicitly.
If $\Gamma\in\mathcal{W}(m)$, then $\Gamma$ defines a Type~$\mathcal{A}$ connection $\nabla^\Gamma$ on $\mathbb{R}^m$;
$\Gamma\in\mathcal{Z}(m)\subset\mathcal{W}(m)$ if and only $\nabla^\Gamma$ is torsion free.  If $\Gamma\in\mathcal{W}(m)$ and $\tilde\Gamma\in\mathcal{W}(m)$,
introduce the equivalence relation $\Gamma\sim\tilde\Gamma$ if there exists the germ of an orientation preserving
diffeomorphism $\Phi$ from a point $P$ in $\mathbb{R}^m$ to a point $\tilde P$ in $\mathbb{R}^m$
so that $\Phi^*(\nabla^{\tilde\Gamma})=\nabla^{\Gamma}$; the precise points in question are irrelevant as the structures are
homogeneous. Since the torsion free condition is preserved by diffeomorphism, $\sim$ defines an equivalence relation on $\mathcal{Z}(m)$ as well.

We say that $\mathcal{M}=(M,\nabla)$ is {\it Type~$\mathcal{A}$} if there is an atlas
$\{\mathcal{U}_\alpha=(U_\alpha,\Gamma_\alpha),\Phi_{\alpha_1\alpha_2}\}$ where
$\{U_\alpha,\Phi_{\alpha_1\alpha_2}\}$ forms an oriented coordinate atlas for $M$ (i.e. $\det(\Phi_{\alpha_1\alpha_2})>0$)
and where $\nabla$ is defined by $\Gamma_\alpha\in\mathcal{W}(m)$ on $U_\alpha$. The coordinate transformations
$\{\Phi_{\alpha_1\alpha_2}\}$ satisfy the intertwining rule
$\Phi_{\alpha_1\alpha_2}\nabla^{\Gamma_{\alpha_2}}=\nabla^{\Gamma_{\alpha_1}}\Phi_{\alpha_1\alpha_2}$.
Note that
$\mathcal{M}$ is {\it torsion free} if and only if $\Gamma_\alpha\in\mathcal{Z}(m)$ for all $\alpha$.
The intertwining rule implies that
$\Gamma_{\alpha_i}\sim\Gamma_{\alpha_j}$ for all $i$ and $j$. We wish study the moduli space of local isomorphism types.
If $\tilde{\mathcal{M}}$ is another Type~$\mathcal{A}$ manifold which is defined by an atlas
$\{(\tilde U_\beta,\tilde\Gamma_\beta),\tilde\Phi_{\beta_1,\beta_2}\}$, then $\mathcal{M}$ and $\tilde{\mathcal{M}}$ are locally isomorphic if and only if
$\Gamma_\alpha\sim\tilde\Gamma_\beta$ for all $\alpha,\beta$. The moduli space of such local isomorphism classes
is then the quotient of $\mathcal{W}(m)$ by the equivalence relation $\sim$. We define:
$$
\mathfrak{W}^+(m):=\mathcal{W}(m)/\sim\text{ and }\mathfrak{Z}^+(m):=\mathcal{Z}(m)/\sim\,.
$$

If $\Gamma\in\mathcal{W}(m)$, then
Equation~(\ref{E1.a}) shows that the Ricci tensor associated to $\Gamma$ is
\begin{equation}\label{E1.b}
\rho_{\Gamma,jk}=\Gamma_{in}{}^i\Gamma_{jk}{}^n-\Gamma_{jn}{}^i\Gamma_{ik}{}^n\,.
\end{equation}
For generic $\Gamma\in\mathcal{W}(m)$, $\rho_{\Gamma,jk}\ne\rho_{\Gamma,kj}$ so $\rho_\Gamma$ is in
general not symmetric. One defines, therefore, the {\it symmetric Ricci tensor} by setting:
$$
\rho_{s,\Gamma}(\eta,\zeta):=\textstyle\frac12(\rho(\eta,\zeta)+\rho(\zeta,\eta))\text{ i.e. }
\rho_{s,\Gamma}:=\textstyle\frac12\{\rho_{\Gamma,jk}+\rho_{\Gamma,kj}\}dx^j\otimes dx^k\,.
$$
If  $\Gamma\in\mathcal{Z}(m)$, then
Equation~(\ref{E1.b}) shows $\rho_{\Gamma,jk}=\rho_{\Gamma,kj}$ so the Ricci tensor is already symmetric and there is
no need to symmetrize. We shall be interested in the case that $\rho_{s,\Gamma}$
is non-degenerate as this is the generic case, see Theorem~\ref{T1.6} below.
Let $\operatorname{sign}(\rho_{s,\Gamma})=(p,q)$ be the signature of the symmetric Ricci tensor; there are $p$ timelike
directions and $q$ spacelike directions. Thus $(p,q)=(m,0)$ implies $\rho_{s,\Gamma}$ is negative definite while
$(p,q)=(0,m)$ implies $\rho_{s,\Gamma}$ is positive definite. If $p+q=m$, i.e. $\rho_{s,\Gamma}$ is non-degenerate, set
$$\begin{array}{ll}
\mathcal{W}(p,q):=\{\Gamma\in\mathcal{W}(m):\operatorname{sign}(\rho_{s,\Gamma})=(p,q)\},&
\mathfrak{W}^+(p,q)\,:=\mathcal{W}(p,q)/\sim\\[0.05in]
\mathcal{Z}(p,q)\ :=\{\Gamma\in\mathcal{Z}(m):\ \operatorname{sign}(\rho_{s,\Gamma})=(p,q)\},&
\mathfrak{Z}^+(p,q):=\mathcal{Z}(p,q)/\sim.
\end{array}$$

\subsection{Reduction to the action of the general linear group}
Let $\operatorname{GL}^+(m,\mathbb{R})$ be the group of linear transformations of $\mathbb{R}^m$ which
preserve the orientation, i.e. $\det(T)>0$. This group acts on the Christoffel symbols $\Gamma\in\mathcal{W}(m)$
of a Type~$\mathcal{A}$ geometry by change of coordinates; two indices are down and one is up.
If $\{e_i\}$ is a basis for $\mathbb{R}^m$ and if $T\in\operatorname{GL}(m,\mathbb{R})$, then
$$
(T\Gamma)(e_i,e_j,e^k):=\Gamma(Te_i,Te_j,Te^k)\,.
$$
One has the following observation \cite{BGGP16a} which shows that in fact one does not need to consider arbitrary
diffeomorphisms in defining the moduli space if the symmetric Ricci tensor is non-degenerate
as the diffeomorphisms $\Phi_{\alpha_1\alpha_2}$ in the atlas are affine. This (in principal)
reduces the problem to one in group representation theory.
\begin{theorem}\label{T1.3}
Let $\mathcal{U}_\alpha=\{(U_\alpha,\Gamma_\alpha),\Phi_{\alpha_1\alpha_2}\}$ be an oriented Type~$\mathcal{A}$ atlas on a
Type~$\mathcal{A}$ manifold $\mathcal{M}$. Assume that the Ricci tensor $\rho_{s,\mathcal{M}}$ is non-degenerate.
\begin{enumerate}
\item $\Phi_{\alpha_1\alpha_2}\vec x_{\alpha_2}= A_{\alpha_1\alpha_2}\vec x_{\alpha_2}+\vec b_{\alpha_1\alpha_2}$ where
$A_{\alpha_1\alpha_2}\in\operatorname{GL}^+(m,\mathbb{R})$ and $\vec b_{\alpha_1\alpha_2}\in\mathbb{R}^m$.
\item
$\mathfrak{W}^+(p,q)=\mathcal{W}(p,q)/\operatorname{GL}^+(m,\mathbb{R})$ and
$\mathfrak{Z}^+(p,q)=\mathcal{Z}(p,q)/\operatorname{GL}^+(m,\mathbb{R})$.
\end{enumerate}\end{theorem}
\begin{proof} The symmetric Ricci tensor is an invariantly
defined pseudo-Riemannian metric on $\mathcal{M}$ which is preserved by the Type~$\mathcal{A}$ coordinate transformations
$\Phi_{\alpha_1\alpha_2}$.
Since $\Gamma$ is constant, the components of $\rho_{s,\Gamma}$ are constant on $U_\alpha$ for any $\alpha$. Thus
$\rho_{s,\Gamma}$ is flat and the coordinate
transformations have the form given. This establishes Assertion~1; as translations do not change $\Gamma$, only the action of
$\operatorname{GL}^+(m,\mathbb{R})$ is relevant in examining the moduli spaces. Assertion~2 now follows.
\end{proof}

\begin{remark}\rm We note that Theorem~\ref{T1.3} fails if we do not assume the Ricci tensor is non-degenerate.
We refer, for example, to \cite{BGGP16a} for a further discussion of this point in the torsion free setting when $m=2$.
\end{remark}

The remainder of this paper is devoted to the study of
these geometries. Since case of surfaces is dealt with in \cite{BGGP16,BGGP16a,G16}, we shall assume for the remainder
of this paper that $m=p+q\ge3$; there are phenomena in this setting not found in the case $m=2$.

\subsection{Principal bundles} Let $G$ be a Lie group which acts smoothly on a manifold $N$.
Let $G_P:=\{g\in G:gP=P\}$ be the {\it isotropy group} of the action.  The action is said to be
{\it fixed point free} if $G_P=\{\operatorname{id}\}$ for all $P$.
The action is said to be {\it proper} if given points $P_n\in N$ and $g_n\in G$ with $P_n\rightarrow P\in N$ and
$g_nP_n\rightarrow\tilde P\in N$, we can
choose a convergent subsequence so $g_{n_k}\rightarrow g\in G$. We refer to \cite{Bo75,GHL04} for the proof of the following result;
see also the discussion in \cite{G16}.

\begin{theorem}\label{T1.5}
Let the action of a Lie group $G$ on a manifold $N$ be fixed point free, smooth, and proper.
Then there is a natural smooth structure on the quotient space $N/G$ so that $G\rightarrow N\rightarrow N/G$ is a principal $G$ bundle.
\end{theorem}

\subsection{Generic phenomena} Let $\mathfrak{P}_m=\mathfrak{P}_m(\Gamma)$ be a polynomial defined on $\mathcal{W}(m)$
which is divisible by $\det(\rho_{s,\Gamma})$ and
which doesn't vanish identically on $\mathcal{Z}(m)$. Let
\smallbreak\qquad$\mathcal{W}(p,q;\mathfrak{P}_m):=\{\Gamma\in\mathcal{W}(p,q):\mathfrak{P}_m(\Gamma)\ne0\}$,
\smallbreak\qquad$\mathfrak{W}^+(p,q;\mathfrak{P}_m):=\mathcal{W}(p,q;\mathfrak{P}_m)/\operatorname{GL}^+(m,\mathbb{R})$,
\smallbreak\qquad$\mathcal{Z}(p,q;\mathfrak{P}_m):=\{\Gamma\in\mathcal{Z}(p,q):\mathfrak{P}_m(\Gamma)\ne0\}$,
\smallbreak\qquad$\mathfrak{Z}^+(p,q;\mathfrak{P}_m):=\mathcal{Z}(p,q;\mathfrak{P}_m)/\operatorname{GL}^+(m,\mathbb{R})$;
\smallbreak\noindent $\mathcal{W}(p,q;\mathfrak{P}_m)$ and $\mathcal{Z}(p,q;\mathfrak{P}_m)$ are open dense subsets of
$\mathcal{W}(p,q)$ and $\mathcal{Z}(p,q)$, respectively.
We will prove the following result in Section~\ref{S2}.
\begin{theorem}\label{T1.6}
There exists a polynomial $\mathfrak{P}_m$ so that
$\mathcal{W}(p,q;\mathfrak{P}_m)$ and $\mathcal{Z}(p,q;\mathfrak{P}_m)$ are $\operatorname{GL}^+(m,\mathbb{R})$ invariant
subsets on which $\operatorname{GL}^+(m,\mathbb{R})$ acts properly and without fixed points.
Consequently, there are natural smooth structures on the moduli spaces
$\mathfrak{W}^+(p,q;\mathfrak{P}_m)$ and
$\mathfrak{Z}^+(p,q;\mathfrak{P}_m)$  so the projections
$\mathcal{W}(p,q;\mathfrak{P}_m)\rightarrow\mathfrak{W}^+(p,q;\mathfrak{P}_m)$
and $\mathcal{Z}(p,q;\mathfrak{P}_m)\rightarrow\mathfrak{Z}^+(p,q;\mathfrak{P}_m)$ are smooth principal
$\operatorname{GL}^+(m,\mathbb{R})$ bundles.
\end{theorem}

\subsection{Results concerning the isotropy subgroup} Let $G_\Gamma^+$ be the group of orientation preserving symmetries of
$(\mathbb{R}^m,\nabla^\Gamma)$; $G_\Gamma^+:=\{T\in\operatorname{GL}^+(m,\mathbb{R}):T\Gamma=\Gamma\}$.
We will prove the following result in Section~\ref{S3}.

\begin{theorem}\label{T1.7}
\ \begin{enumerate}
\item Let $\Gamma_n\in\mathcal{W}(p,q)$ satisfy $\Gamma_n\rightarrow\Gamma\in\mathcal{W}(p,q)$.
If $\dim\{G_{\Gamma_n}^+\}\ge1$, then $\dim\{G_\Gamma^+\}\ge1$.
\item There exists $c(m)$ so that if $\Gamma\in\mathcal{W}(p,q)$ and  if no element of $G_\Gamma^+$ has infinite order, then every element in
$G_\Gamma^+$ has order at most $c(m)$. Furthermore, $\lim_{m\rightarrow\infty}c(m)=\infty$.
\end{enumerate}
\end{theorem}

\subsection{Definite symmetric Ricci tensor } Let
$$\begin{array}{ll}
\tilde{\mathcal{W}}(p,q):=\left\{\Gamma\in\mathcal{W}(p,q):G_\Gamma^+=\{\operatorname{id}\}\right\},&
\tilde{\mathfrak{W}}^+(p,q):=\tilde{\mathcal{W}}(p,q)/\operatorname{GL}^+(m,\mathbb{R}),\\[0.05in]
\tilde{\mathcal{Z}}(p,q):=\left\{\Gamma\in\mathcal{Z}(p,q):G_\Gamma^+=\{\operatorname{id}\}\right\},&
\tilde{\mathfrak{Z}}^+(p,q):=\tilde{\mathcal{Z}}(p,q)/\operatorname{GL}^+(m,\mathbb{R}).
\end{array}$$
If $(p,q)\in\{(0,m),(m,0)\}$ so $\rho_{s,\Gamma}$ is definite, then one need not consider the generic situation
but can simply exclude the fixed point sets and work directly with the sets $\tilde{\mathcal{W}}(p,q)$ and $\tilde{\mathcal{Z}}(p,q)$.
 We shall prove the following result in Section~\ref{S4}.

\begin{theorem}\label{T1.8}
Let $(p,q)\in\{(m,0),(0,m)\}$. Then:
\begin{enumerate}
\item The action of $\operatorname{GL}^+(m,\mathbb{R})$ on ${\mathcal{W}}(p,q)$ and on ${\mathcal{Z}}(p,q)$ is proper.
\item $\tilde{\mathcal{W}}(p,q)$ and $\tilde{\mathcal{Z}}(p,q)$ are open dense subsets of $\mathcal{W}(p,q)$ and
$\mathcal{Z}(p,q)$, respectively.
\item One can define natural smooth structures on the associated moduli spaces
$\tilde{\mathfrak{W}}^+(p,q)$ and
$\tilde{\mathfrak{Z}}^+(p,q)$
so that  $\tilde{\mathcal{W}}(p,q)\rightarrow\tilde{\mathfrak{W}}^+(p,q)$ and $\tilde{\mathcal{Z}}(p,q)\rightarrow\tilde{\mathfrak{Z}}^+(p,q)$
are smooth principal $\operatorname{GL}^+(m,\mathbb{R})$ bundles.\end{enumerate}\end{theorem}

\subsection{The higher signature setting}
In Section~\ref{S5}, we will prove the following result which shows that Assertion~1 of Theorem~\ref{T1.8} fails in
 the higher signature setting.

\begin{theorem}\label{T1.9}
Let $p\ge1$, let $q\ge1$, and let $p+q\ge3$.
 There exists $\Gamma\in\mathcal{Z}(p,q)$ so that
$G_\Gamma^+$ is not compact.  Consequently, the action of $\operatorname{GL}^+(m,\mathbb{R})$ on $\mathcal{Z}(p,q)$
or on $\mathcal{W}(p,q)$ is not proper.\end{theorem}

\subsection{Two dimensional geometry}
The two dimensional setting is relatively easy to examine. Since it informs many of the constructions we will employ
in the 3-dimensional setting, it seems worth while discussing it in a bit of detail; a construction which will be used in the proof
of Theorem~\ref{T1.7}~(2) will renter in analysis of the 2-dimensional setting. We introduce the following basic structure:
\begin{definition}\label{D1.10}\rm
Let $\Gamma_2$ be the structure
$$\begin{array}{llllll}
\Gamma_{11}{}^1=\frac1{\sqrt2},&\Gamma_{11}{}^2=0,&
\Gamma_{12}{}^1=0,&\Gamma_{12}{}^2=-\frac1{\sqrt2},&
\Gamma_{22}{}^1=-\frac1{\sqrt2},&
\Gamma_{22}{}^2=0.
\end{array}$$
We obtain $\rho_\Gamma=\operatorname{diag}(-1,-1)$.
\end{definition}

The following result was proved in \cite{BGGP16a} using different methods;  we give a different
proof in Section~\ref{S6} to introduce arguments we will use subsequently.

\begin{theorem}\label{T1.11}
Adopt the notation established above.
\begin{enumerate}
\item Let $(p,q)=(1,1)$ or $(p,q)=(0,2)$. Then the action of $\operatorname{GL}^+(2,\mathbb{R})$ on $\mathcal{Z}(p,q)$ is
fixed point free and proper. Thus $\mathcal{Z}(p,q)\rightarrow\mathfrak{Z}^+(p,q)$
is a principal $\operatorname{GL}^+(2,\mathbb{R})$ bundle over a real analytic surface.
\item Let $(p,q)=(2,0)$. If $G_\Gamma^+\ne\{\id \}$, then
$G_\Gamma^+=\mathbb{Z}_3$ and
$\Gamma$ is isomorphic to the structure $\Gamma_2$ of Definition~\ref{D1.10}. $\operatorname{GL}^+(2,\mathbb{R})$
acts properly on $\mathcal{Z}(2,0)$ and
$\mathcal{Z}(2,0)\rightarrow\mathfrak{Z}_{2,0}^+-[\Gamma_2]$ is a principal $\operatorname{GL}(2,\mathbb{R})$ bundle
over a real analytic surface once we remove the exceptional orbit corresponding to $\Gamma_2$.
\end{enumerate}
\end{theorem}

\subsection{Three dimensional geometry}
We now restrict to dimension $m=3$ and the torsion free setting in order to illustrate the
possible isotropy subgroups. The examples where $\dim\{G_\Gamma^+\}>0$ form two families given in (1) and (2) below.
The remaining structure groups are all finite and comprise one of the following
$\{\mathbb{Z}_3$, $\mathbb{Z}_2\oplus\mathbb{Z}_2$, $\mathbb{Z}_2$,
$s_3$, $a_4\}$; they appear in the two families given in (3) and (4) below. Thus one obtains that the constant $c(3)=3$ in Theorem~\ref{T1.7}.
We will prove the following result in Section~\ref{S7}:

\begin{theorem}\label{T1.12}
Let $\Gamma\in\mathcal{Z}(p,q)$ for $p+q=3$.
Assume $G_\Gamma^+\ne\{\operatorname{id}\}$.
We can make a linear change of coordinates so that one of the following 4 possibilities holds:
\begin{enumerate}
\item There exist $(a,b,c,d)\in\mathbb{R}^4$ so
$G_\Gamma^+=\operatorname{SO}(1,1)$,
$\Gamma_{12}{}^3=a$, $\Gamma_{13}{}^1=b$,
$\Gamma_{23}{}^2=c$, $\Gamma_{33}{}^3=d$, and
$\rho=ad(e^1\otimes e^2+e^2\otimes e^1)+(-b^2+bd+c(-c+d))e^3\otimes e^3$.
We require $ad\ne0$ and $-b^2+bd+c(-c+d)\ne0$.
\item There exist $(a,b,c,d)\in\mathbb{R}^4$ so
$G_\Gamma^+=\operatorname{SO}(2)$,
$\Gamma_{11}{}^3=a$, $\Gamma_{13}{}^1=b$, $\Gamma_{13}{}^2=c$, $\Gamma_{22}{}^3=a$,
$\Gamma_{23}{}^1=-c$, $\Gamma_{23}{}^2=b$, $\Gamma_{33}{}^3=d$,
$\rho_\Gamma=\operatorname{diag}(ad,ad,2(bd-b^2+c^2))$. We require
$ad\ne0$ and $bd-b^2+c^2\ne0$.
\item The group $G_\Gamma^+$ is finite, there exists an element of order $3$ in $G_\Gamma^+$, and
there exist $(a,b,c,d)\in\mathbb{R}^4$ so
$\Gamma_{11}{}^1=1$, $\Gamma_{11}{}^3=a$, $\Gamma_{12}{}^2=-1$, $\Gamma_{13}{}^1=b$, $\Gamma_{13}{}^2=c$,
$\Gamma_{22}{}^1=-1$, $\Gamma_{22}{}^3=a$, $\Gamma_{23}{}^1=-c$, $\Gamma_{23}{}^2=b$, $\Gamma_{33}{}^3=d$, and
$\rho_\Gamma=(ad-2)(e^1\otimes e^1+e^2\otimes e^2+2(bd-b^2+c^2))e^3\otimes e^3$. We require $ad-2\ne0$ and
$-b^2+c^2+bd\ne0$.
We have $G_\Gamma^+=\mathbb{Z}_3$ except for the following exceptional structures which are given up to isomorphism by:
\begin{enumerate}\item
 $a=0$, $b=0$, $c=1$, $d=0$, and $G_\Gamma^+=s_3$.
\item $c=0$, $a=b=\pm\frac1{\sqrt2}$, $d=\pm\sqrt{2}$,
and $G_\Gamma^+=a_4$.
\end{enumerate}\item
The group $G_\Gamma^+$ is finite and all elements of $G_\Gamma^+$ have order 2. There are two structures up to isomorphism:
\begin{enumerate}
\item
$G_\Gamma^+=\mathbb{Z}_2\oplus\mathbb{Z}_2$,
$\Gamma_{12}{}^3=1$, $\Gamma_{13}{}^2=1$, $\Gamma_{23}{}^1=-1$, and\newline
$\rho=-2(e^1\otimes e^1+e^2\otimes e^2)+2e^3\otimes e^3$.
\item $G_\Gamma^+=\mathbb{Z}_2$, $\Gamma_{ij}{}^k=0$ unless the index $3$ appears an odd number
of times,
$\Gamma_{11}{}^3 = a$, $\Gamma_{12}{}^3 = b$, $\Gamma_{13}{}^1 = c$, $\Gamma_{13}{}^2 = d$,
$\Gamma_{21}{}^3 = b$, $\Gamma_{22}{}^3 = e$, $\Gamma_{23}{}^1 = f$, $\Gamma_{23}{}^2 = g$,
$\Gamma_{31}{}^1 = c$, $ \Gamma_{31}{}^2 = d$, $\Gamma_{32}{}^1 = f$, $\Gamma_{32}{}^2 = g$,
$\Gamma_{33}{}^3 = h$.
$\rho_{11}=-2 b d + a (-c + g + h)$,\quad
$\rho_{12}=\rho_{21}=-d e - a f + b h$,
$\rho_{33}=-c^2 - 2 d f + c h + g (-g + h)$. One requires $\det(\rho)\ne0$.
\end{enumerate}\end{enumerate}\end{theorem}

\subsection{The unoriented category} There are similar results in the unoriented category. One does not
assume $\mathcal{M}$ is oriented and one replaces the structure
group $\operatorname{GL}^+(m,\mathbb{R})$ by the full general linear group. Theorems~\ref{T1.6}--\ref{T1.9}
extend to this context with only the appropriate minor modifications of notation. The corresponding analysis of Theorem~\ref{T1.12}
in dimension $3$
would become much more complicated and we have not attempted it for that reason nor have we considered torsion in these results
for the same reason as our purpose was to be illustrative rather than exhaustive. We have chosen to work in the smooth category;
however all the structures in question and the relevant morphisms are in fact real analytic.

\section{Generic properties}\label{S2}
We introduce the following tensors. Let
\begin{equation}\label{E2.a}
\omega:=\Gamma_{ij}{}^jdx^i,\quad
\rho_{1,\Gamma}=\Gamma_{in}{}^i\Gamma_{jk}{}^ndx^j\otimes dx^k,\text{ and }\rho_{2,\Gamma}=\Gamma_{jn}{}^i\Gamma_{ik}{}^ndx^j\otimes dx^k\,.
\end{equation}
Note that $\rho=\rho_1-\rho_2$. If $\rho_{s,\Gamma}$ is non-degenerate, then
$\rho_{s,\Gamma}(\varepsilon):=\rho_{s,\Gamma}+\varepsilon\rho_{2,s,\Gamma}$ is invertible for small $\varepsilon$.
Let $\varrho_{s,\Gamma}^{i\ell}(\varepsilon)$ be the components of the inverse matrix; this defines the dual symmetric non-degenerate
2-tensor
on $(\mathbb{R}^m)^*$. As $\rho_{s,\Gamma}(\varepsilon)$ is real analytic in $\varepsilon$, we sum over repeated indices to expand
$$
\rho_{s,\Gamma}^{i\ell}(\varepsilon)\Gamma_{ij}{}^je_\ell=\sum_{n=0}^\infty\xi_{\Gamma,n}\varepsilon^n
\text{ where }\xi_{\Gamma,n}\in\mathbb{R}^m\,.
$$
We begin the proof of Theorem~\ref{T1.6} with the following observation.
\begin{lemma}\label{L2.1} There exists a polynomial $\mathfrak{P}_m=\mathfrak{P}_m(\Gamma)$ and an integer $\kappa_m$ so that:
\begin{enumerate}
\item If $\Gamma\in \mathcal{W}(m)$ and if $\mathfrak{P}_m(\Gamma)\ne0$, then
\begin{enumerate}\item $\rho_{s,\Gamma}$ is non-degenerate,
\item $\mathcal{B}_\Gamma:=\{\xi_{\Gamma,0},\xi_{\Gamma,1},\dots,\xi_{\Gamma,m-1}\}$ is a basis for $\mathbb{R}^m$.
\item $G_\Gamma^+=\{\operatorname{id}\}$.
\end{enumerate}
\item If $T\in\operatorname{GL}^+(m,\mathbb{R})$ and if $\Gamma\in\mathcal{W}(m)$, then
$\mathfrak{P}_m(\Gamma)=\det(T)^{\kappa(m)}\mathfrak{P}_m(T\Gamma)$.
\item There exists $\Gamma\in \mathcal{Z}(m)$ so that $\mathfrak{P}_m(\Gamma)\ne0$.
\end{enumerate}
\end{lemma}

\begin{proof}
Clearly $\rho_{s,\Gamma}$ is non-degenerate if and only if $\det(\rho_{s,\Gamma})\ne0$. So
we will make $\det(\rho_{s,\Gamma})$ a factor of our polynomial to ensure that Assertion~(1a) is valid.
We apply Cramer's rule. Let $\tilde\rho_{s,\Gamma}$ be the matrix of cofactors of $\rho_{s,\Gamma}$; this is a matrix valued polynomial
which is well defined for all $\Gamma\in\mathcal{Z}(m)$ such that if $\det(\rho_{s,\Gamma})\ne0$, then $\rho_{s,\Gamma}^{-1}=\det(\rho_{s,\Gamma})^{-1}\tilde\rho_{s,\Gamma}$.
Suppose $\rho_{s,\Gamma}$ is invertible. We use the Neumann series to expand
\begin{eqnarray*}
&&(\rho_{s,\Gamma}+\varepsilon\rho_{2,s,\Gamma})^{-1}
=\{(\rho_{s,\Gamma}(\operatorname{id}+\varepsilon\rho_{s,\Gamma}^{-1}\rho_{2,s,\Gamma})\}^{-1}
=(\operatorname{id}+\varepsilon\rho_{s,\Gamma}^{-1}\rho_{2,s,\Gamma})^{-1}\rho_{s,\Gamma}^{-1}\\
&=&\sum_{n=0}^\infty(-1)^n(\rho_{s,\Gamma}^{-1}\rho_{2,s,\Gamma})^n\rho_{s,\Gamma}^{-1}\varepsilon^n
=\sum_{n=0}^\infty(-1)^n\det(\rho_{s,\Gamma})^{-n-1}(\tilde\rho_{s,\Gamma}\rho_{2,s,\Gamma})^n\tilde\rho_{s,\Gamma}\varepsilon^n\,.
\end{eqnarray*}
We let $\tilde\xi_n:=(\tilde\rho_{s,\Gamma}\rho_{2,s,\Gamma})^n\tilde\rho_{s,\Gamma}\omega$. We then have
\begin{equation}\label{E2.b}
\xi_{\Gamma,n}=(-1)^n\det(\rho_{s,\Gamma})^{-n-1}\tilde\xi_{\Gamma,n}\,.
\end{equation}
If $\det(\rho_{s,\Gamma})\ne0$, then $\{\xi_{\Gamma,0},\dots,\xi_{\Gamma,n}\}$ is a basis for $\mathbb{R}^m$ if and only if
$\{\tilde\xi_{\Gamma,0},\dots,\tilde\xi_{\Gamma,n}\}$ is a basis for $\mathbb{R}^m$. We define a polynomial $\mathfrak{P}_m$
which satisfies (1a) and (1b) by defining:
$$
\mathfrak{P}_m(\Gamma):=\det(\rho_{s,\Gamma})\det(\tilde\xi_{0,\Gamma},\dots,\tilde\xi_{m,\Gamma})\,.
$$

Since contracting an upper index against a lower index is invariant under the action of $\operatorname{GL}^+(m,\mathbb{R})$,
Theorem~\ref{T1.3} shows the tensors of Equation~(\ref{E2.a}) are invariantly defined; if $T\in\operatorname{GL}^+(m,\mathbb{R})$
and $\Gamma\in\mathcal{W}(m)$, one has that:
\begin{equation}\label{E2.c}\begin{array}{lll}
\omega(T\Gamma)=T\omega(\Gamma),&
\rho_1(T\Gamma)=T\rho_1(\Gamma),&
\rho_2(T\Gamma)=T\rho_2(\Gamma),\\[0.05in]
\rho_{2,s}(T\Gamma)=T\rho_{2,s}(\Gamma),&
T\xi_{\Gamma,i}=\xi_{T\Gamma,i},&
T\mathcal{B}_\Gamma=\mathcal{B}_{T\Gamma}.
\end{array}\end{equation}
 Let $T\in G_\Gamma^+$ with $\mathfrak{P}_m(\Gamma)\ne0$. Then $\mathcal{B}_\Gamma$ is a basis for
$\mathbb{R}^m$. Since $T\mathcal{B}_\Gamma=\mathcal{B}_{T\Gamma}=\mathcal{B}_\Gamma$,
$T=\operatorname{id}$. Assertion~(1c) now follows.

We may verify Assertion~2 as follows. We have
\begin{equation}\label{E2.d}
\det(\rho_{s,T\Gamma})=\det(T\rho_{s,\Gamma})=\det(T)^2\det(\rho_{s,\Gamma})\,.
\end{equation}
In particular, if $\det(\rho_{s,\Gamma})=0$ then $\det(\rho_{s,T\Gamma})=0$ and Assertion~2 holds trivially.
Suppose $\rho_{s,\Gamma}$ is non-singular. Let $c(m):=(1+2+\dots+(m-1))+1$ and let $\kappa(m)=2c(m)+m+2$. Since
$T\xi_{\Gamma,i}=\xi_{T\Gamma,i}$, we may verify Assertion~2 by using Equation~(\ref{E2.b}) and Equation~(\ref{E2.d}) to compute:
\begin{eqnarray*}
\mathfrak{P}_m(T\Gamma)&=&\det(\rho_{s,T\Gamma})\det(\tilde\xi_{T\Gamma,0},\dots,\tilde\xi_{T\Gamma,m-1})\\
&=&\det(\rho_{s,T\Gamma})^{c(m)+1}\det(\xi_{T\Gamma,0},\dots,\xi_{T\Gamma,m-1})\\
&=&\det(T)^{2c(m)+2}\det(\rho_{s,\Gamma})^{c(m)+1}\det(T\xi_{\Gamma,0},\dots,T\xi_{\Gamma,m-1})\\
&=&\det(T)^{2c(m)+m+2}\det(\rho_{s,\Gamma})^{c(m)+1}\det(\xi_{\Gamma,0},\dots,\xi_{\Gamma,m-1})\\
&=&\det(T)^{2c(m)+m+2}\det(\rho_{s,\Gamma})\det(\tilde\xi_{\Gamma,0},\dots,\tilde\xi_{\Gamma,m-1})\\
&=&\det(T)^{2c(m)+m+2}\mathfrak{P}_m(\Gamma)\,.
\end{eqnarray*}

We complete the proof by exhibiting torsion free Christoffel symbols in all dimensions where
$\rho_{s,\Gamma}$ is non-degenerate and where $\{\Xi_{\Gamma,0},\dots,\Xi_{\Gamma,m-1}\}$ are a basis for $\mathbb{R}^m$.
We proceed by considering various cases.
\medbreak\noindent{\bf Case 1.} Let $m=2$.
We use the parametrization of \cite{BGGP16a} and define a torsion free tensor by setting $\Gamma(x)_{ij}{}^k$
$$
\begin{array}{llll}
\Gamma(x)_{11}{}^1=x+\frac1x,&\Gamma(x)_{11}{}^2=0,&\Gamma(x)_{12}{}^1=0&\Gamma(x)_{21}{}^1=0,\\
\Gamma(x)_{22}{}^1=x,&\Gamma(x)_{12}{}^1=x,&\Gamma(x)_{21}{}^2=x,&\Gamma(x)_{22}{}^2=1.
\end{array}$$
We may then compute:
$$
\rho_{2,\Gamma}(x)=\left(\begin{array}{cc}2+\frac1{x^2}+2x^2&x\\x&1+2x^2\end{array}\right),\
\rho_{s,\Gamma}(x)=\left(\begin{array}{cc}1&0\\0&1\end{array}\right),\
\omega=\left(\begin{array}{c}x+\frac1x\\1\end{array}\right)\,.
$$
In particular $\rho_{s,\Gamma}$ is non-singular. As
$(\rho_{s,\Gamma}+\varepsilon\rho_{2,s,\Gamma})^{-1}=\operatorname{id}-\varepsilon\rho_{2,s,\Gamma}+O(\varepsilon^2)$, one has that
$$
\xi_0(x)=\left(\begin{array}{c}x+\frac1x\\1\end{array}\right),\quad
\xi_1(x)=\left(\begin{array}{c}\frac1{x^3}+\frac3x+5x+2x^3\\2+3x^2\end{array}\right)\,.
$$
Choose $x$ so $\xi_0(x)$ and $\xi_1(x)$ are linearly independent to complete the proof if $m=2$.
\medbreak\noindent{\bf Case 2.} Suppose $m=2\bar m$ for $m\ge2$. We let
$\Gamma(\vec x):=\Gamma(x_1)\oplus\dots\oplus\Gamma(x_{\bar m})$.
The structures decouple;
\begin{eqnarray*}
&&\rho_{s,\Gamma}(\vec x)=\rho_{s,\Gamma}(x_1)\oplus\dots\oplus\rho_{s,\Gamma}(x_{\bar m})=\operatorname{id},\quad
\rho_2(\vec x)=\rho_2(x_1)\oplus\dots\oplus\rho_2(x_{\bar m}),\\
&&(\rho_{s,\Gamma}(\vec x)+\varepsilon\rho_2(\vec x))^{-1}=\sum_{n=0}^\infty(-1)^n\varepsilon^n\rho_2(\vec x)^n.
\end{eqnarray*}
Since $\rho_{s,\Gamma}=\operatorname{id}$, $\rho_{s,\Gamma}$ is non-singular. Let
$V:=\operatorname{Span}_{0\le n\le m-1}\{\rho_2(\vec x)^n\omega(\vec x)\}$.
We must show $V=\mathbb{R}^m$. As the minimal polynomial of $\rho_2(\vec x)$ has degree at most $m-1$,
it is not necessary to truncate by taking $n\le m-1$ and we have
$$V=\operatorname{Span}_{0\le n}\{p_2(\vec x)^n\omega(\vec x)\}\,.$$
We assume $0<x_1<\dots<x_{\bar m}$. As $n\rightarrow\infty$, the terms in $x_{\bar m}$
will dominate. We examine the final block
$$
\rho_2(\vec x_{\bar m})^n=(2x_{\bar m}^2)^n\left(\begin{array}{cc}1&0\\0&1\end{array}\right)^n
\left(\begin{array}{c}x_{\bar m}\\0\end{array}\right)+O(x_{\bar m}^{2n})\,.
$$
The other blocks do not play a role so $\lim_{n\rightarrow\infty}\{(2x_{\bar m}^2)^{-n}x_{\bar m}^{-1}\}\omega=e_{2m-1}$.
Since $V$ is a closed $\mathbb{Z}(\rho_2(\vec x))$ module, $e_{2m-1}\in V$. Examining
$\omega-(\frac1{x_{\bar m}}+x_{\bar m})e_{2m-1}$ and applying a similar argument to the last block yields as well $e_{2m}\in V$. We can
now work our way backwards through the blocks to see $V=\mathbb{R}^m$ as desired.
\medbreak\noindent{\bf Case 3.} Suppose $m=3+2k$ for $k\ge0$ is odd. Applying exactly the same asymptotic analysis
as used in the even dimensional case, we are reduced to considering the case $m=3$. We set
$$\begin{array}{llll}
\Gamma_{11}{}^1=2,&\Gamma_{22}{}^2=4,&\Gamma_{33}{}^3=2,\\
\Gamma_{11}{}^3=1,& \Gamma_{13}{}^1=1,&\Gamma_{31}{}^1=1,\\
\Gamma_{23}{}^2=1,&\Gamma_{32}{}^2=1,&\Gamma_{22}{}^3=1\,.
\end{array}$$
We then compute:
\begin{eqnarray*}
\rho_2=\left(\begin{array}{ccc}6&0&2\\0&18&4\\2&4&6\end{array}\right),\quad
\rho_{s,\Gamma}=\left(\begin{array}{ccc}2&0&0\\0&2&0\\0&0&2\end{array}\right),\quad
\omega=\left(\begin{array}{c}2\\4\\4\end{array}\right).
\end{eqnarray*}
Let $A:=\frac12\rho_2$. We then have
$$
(\rho_{s,\Gamma}+\varepsilon\rho_{2,s,\Gamma})^{-1}=\frac12(\operatorname{id}+A)^{-1}=\frac12\sum_{n=0}^\infty(-1)^n\varepsilon^nA^n\,.
$$
We complete the proof by computing:
\medbreak
$\Xi_0=\left(\begin{array}{c}1\\2\\2\end{array}\right),\
\Xi_1=\left(\begin{array}{c}5\\15\\11\end{array}\right),\
\Xi_2=\left(\begin{array}{c}26\\78\\68\end{array}\right)$,
$\det\left(\begin{array}{ccc}1&5&26\\2&15&78\\2&11&68\end{array}\right)=54$.
\end{proof}
\subsection*{The proof of Theorem~\ref{T1.6}}
We use Lemma~\ref{L2.1} to see $\operatorname{GL}^+(m,\mathbb{R})$ preserves $\mathcal{W}(p,q;\mathfrak{P}_m)$
and $\mathcal{Z}(p,q;\mathfrak{P}_m)$ and that $\operatorname{GL}^+(m,\mathbb{R})$ acts without fixed points on
these sets. Theorem~\ref{T1.6} will then follow from Theorem~\ref{T1.5} if we can show the action is proper.

Given bases $\mathcal{B}$ and $\tilde{\mathcal{B}}$ for $\mathbb{R}^m$, let $T_{\mathcal{B},\tilde{\mathcal{B}}}$
be the unique linear transformation taking $\mathcal{B}$ to $\tilde{\mathcal{B}}$. Let $\Gamma\in\mathcal{Z}(p,q;\mathfrak{P}_m)$.
By Equation~(\ref{E2.c}), $T\mathcal{B}_\Gamma=\mathcal{B}_{T\Gamma}$ so $T=T_{\mathcal{B}_\Gamma,\mathcal{B}_{T\Gamma}}$.
Let
$
\{\Gamma_k,\Gamma,\tilde\Gamma\}\subset\mathcal{W}(p,q;\mathfrak{P}_m)$ and
$T_k\in\operatorname{GL}^+(m,\mathbb{R})$. Assume that $\Gamma_k\rightarrow\Gamma$ and $T\Gamma_k\rightarrow\tilde\Gamma$.
This implies $\mathcal{B}_{\Gamma_k}\rightarrow\mathcal{B}_\Gamma$ and
$\mathcal{B}_{T\Gamma_k}\rightarrow\mathcal{B}_{\tilde\Gamma}$. Consequently,
\smallbreak\hfill$T_k=T_{\mathcal{B}_{\Gamma_k},\mathcal{B}_{T\Gamma_k}}\rightarrow T_{\mathcal{B}_\Gamma,\mathcal{B}_{\tilde\Gamma}}$.\hfill\vphantom{.}
\hfill\qed

\section{The proof of Theorem~\ref{T1.7}}\label{S3}
\subsection{The proof of Theorem~\ref{T1.7}~(1)}
Suppose that
$$
\Gamma_n\in\mathcal{W}(p,q)\text{ and }\Gamma_n\rightarrow\Gamma\in\mathcal{W}(p,q)\,.
$$
Assume $\dim\{G_{\Gamma_n}^+\}\ge1$. We must show $\dim\{G_\Gamma^+\}\ge1$. Since $\dim\{G_{\Gamma_n}^+\}\ge1$,
we may find $0\ne\xi_n\in\operatorname{gl}(m,\mathbb{R})$ so that $\exp(t\xi_n)$ defines a 1-parameter subgroup of
$G_{\Gamma_n}^+$. Let $\|\cdot\|$ be a norm on $\operatorname{gl}(m,\mathbb{R})$. We may assume without loss of generality
that $\|\xi_n\|=1$ and extract a convergent subsequence $\xi_n\rightarrow\xi$ with $\|\xi\|=1$. We then have by continuity
that $\exp(t\xi)$ is a 1-parameter subgroup of $G_\Gamma^+$ and thus, in particular, $\dim\{G_\Gamma^+\}\ge1$.

\subsection{The proof of Theorem~\ref{T1.7}~(2)}
Fix $m\ge3$ and let
$$\vartheta=(\vartheta_{ijk})\text{ for }\vartheta\in\{0,1\}\text{ and }1\le i,j,k\le m\,.
$$
Let $A_\vartheta$ be the Abelian group which is generated multiplicatively by indeterminates $\kappa_1,\dots,\kappa_m$ subject to the relations
$\kappa_i\kappa_j=\kappa_k$ whenever $\vartheta_{ijk}=1$. Let $\operatorname{Tor}(A_\vartheta)$ be the subgroup
of $A$ consisting of all elements of finite order. Let
$$
c(m):=\max_{T\in\operatorname{Tor}(A)}\operatorname{order}(T)\,.
$$

Let $\Gamma\in\mathcal{Z}(p,q)$.
Assume that no element of $G_\Gamma^+$ has infinite order. Let $T\in G_\Gamma^+$.
Since $T$ has finite order,  there exists a complex basis $\{f_1,\dots,f_m\}$
for $\mathbb{C}$ so that $Tf_i=\lambda_if_i$ where $|\lambda_i|=1$.
If we express $\lambda_i=e^{\sqrt{-1}\theta_i}$, then the $\theta_i$ are
the rotation angles of $T$ regarded as a real map and the eigenvalues $\lambda_i$ occur in conjugate pairs for $\theta_i\notin\{0,\pi\}$.
We have $T\Gamma_{ij}{}^k=\lambda_i\lambda_j\lambda_k^{-1}\Gamma_{ij}{}^k$.
Set $\vartheta_{\Gamma,ijk}=1$ if $\Gamma_{ij}{}^k\ne0$ and set $\vartheta_{J,ijk}=0$ otherwise.  Then $\vec\lambda$ can be
regarded as an element of $A_\vartheta$. Furthermore, if $A_\vartheta$ has an element of infinite order, then there exists $\vec\lambda$
where the eigenvalues occur suitably in conjugate pairs
so $T_\lambda \Gamma=\Gamma$ and $T_\lambda$ has infinite order. Since this is false, $A_{\vartheta(\Gamma)}$ is finite
and thus $\operatorname{order}(T_\lambda)\le c(m)$ as desired.

To show that $\lim_{m\rightarrow\infty}c(m)=\infty$, we will construct a family of Type~$\mathcal{A}$ connections
$\Gamma_{3\ell}$ on $\mathbb{R}^{3\ell}$ so that there exists an element $T\in G_{\Gamma_{3\ell}}^+$ which has
order $2^\ell-1$ and such that there is no element of infinite order in $G_{\Gamma_{3\ell}}^+$. We shall work in
the torsion free setting so there is no need to symmetrize. Recall that
$$
\rho=\rho_1-\rho_2\text{ where }\rho_{1;jk}:=\Gamma_{in}{}^i\Gamma_{jk}{}^n\text{ and }
\rho_{2;jk}:=\Gamma_{jn}{}^i\Gamma_{ik}{}^n\,.
$$
Central to our construction is a 3-dimensional example. Let $\{e_1,e_2,e_3\}$ be the standard basis for $\mathbb{R}^3$. Introduce
a complex basis
$$\begin{array}{lll}
f_1:=e_1+\sqrt{-1}e_2,&f_2:=e_1-\sqrt{-1}e_2,&f_3:=e_3,\\
f^1:=\frac12(e^1-\sqrt{-1}e^2),&f^2:=\frac12(e^1+\sqrt{-1}e^2),&f^3:=e^3.
\end{array}$$
For $a>0$, let the non-zero Christoffel symbols be given by:
$$\begin{array}{lll}
\Gamma(f_1,f_3,f^1)=a,&\Gamma(f_3,f_1,f^1)=a,&\Gamma(f_2,f_3,f^2)=a,\\
\Gamma(f_3,f_2,f^2)=a,&\Gamma(f_1,f_2,f_3)=a,&\Gamma(f_2,f_1,f^3)=a,\\
\Gamma(f_3,f_3,f^3)=\frac{a^2+1}a.
\end{array}$$
We may then compute:
\medbreak\qquad$\rho_1=\left(\begin{array}{ccc} 0 & 3 a^2+1 & 0 \\ 3 a^2+1 & 0 & 0 \\ 0 & 0 & 3 a^2+\frac{1}{a^2}+4 \\ \end{array}\right)$,
\medbreak\qquad$\rho_2=\left(\begin{array}{ccc} 0 & 2 a^2 & 0 \\ 2 a^2 & 0 & 0 \\ 0 & 0 & 3 a^2+\frac{1}{a^2}+2 \\\end{array}\right)$,\quad
$\rho=\left(\begin{array}{ccc} 0 & a^2+1 & 0 \\ a^2+1 & 0 & 0 \\ 0 & 0 & 2 \\\end{array}\right)$.
\medbreak\noindent Relative to the underlying real basis $\{e_1,e_2,e_3\}$ we have:
\medbreak\qquad
$\textstyle \rho_1=\left(\begin{array}{ccc}3a^2+1&0&0\\0&3a^2+1&0\\0&0&3 a^2+\frac{1}{a^2}+4\end{array}\right)$,
\medbreak\qquad$\rho_2=\left(\begin{array}{ccc}2a^2&0&0\\0&2a^2&0\\0&0&3 a^2+\frac{1}{a^2}+2\end{array}\right)$,
\quad$\rho=\left(\begin{array}{ccc}a^2+1&0&0\\0&a^2+1&0\\0&0&2\end{array}\right)$.
\medbreak We take a basis $\{e_{1,\mu},e_{2,\mu},e_{3,\mu}\}$ for $\mathbb{R}^{3\ell}$ where $1\le\mu\le\ell$. Let the non-zero
Christoffel symbols be given by:
$$\begin{array}{lll}
\Gamma(f_{1,\mu},f_{3,\mu},f^{1,\mu})=a_\mu,&\Gamma(f_{3,\mu},f_{1,\mu},f^{1,\mu})=a_\mu,&
\Gamma(f_{2,\mu},f_{3,\mu},f^{2,\mu})=a_\mu,\\
\Gamma(f_{3,\mu},f_{2,\mu},f^{2,\mu})=a_\mu,&\Gamma(f_{1,\mu},f_{2,\mu},f_{3,\mu})=a_\mu,&
\Gamma(f_{2,\mu},f_{1,\mu},f^{3,\mu})=a_\mu,\\
\Gamma(f_{3,\mu},f_{3,\mu},f^{3,\mu})=\frac{1+a_\mu^2}{a_\mu},&\Gamma(f_{1,\mu},f_{1,\mu},f^{1,\mu+1})=1,&\Gamma(f_{2,\mu},f_{2,\mu},f^{2,\mu+1})=1.
\end{array}$$
Here we let $\mu$ be defined modulo $\ell$ so $f_{i,\ell+1}=f_{i,1}$. This defines corresponding real Christoffel symbols.
Let $\vec\lambda\in\mathbb{C}^\ell$. We assume $|\lambda_\mu|=1$ so $\lambda_\mu^{-1}=\bar\lambda_\mu$. Let
$$\begin{array}{lll}
T(f_{1,\mu})=\lambda_\mu f_{1,\mu},&T(f_{2,\mu})=\bar\lambda_\mu f_{2,\mu},&T(f_{3,\mu})=f_{3,\mu},\\
T(f^{1,\mu})=\bar\lambda_\mu f^{1,\mu},&T(f^{2,\mu})=\lambda_\mu f^{2,\mu},&T(f^{3,\mu})=f^{3,\mu}.
\end{array}$$
To ensure that $T\Gamma=\Gamma$, we must have $\lambda_\mu^2=\lambda_{\mu+1}$. Setting $\lambda_{\ell+1}=\lambda_1$ then
leads to the relation $\lambda_1^{2^\ell}=\lambda_1$ and thus the cyclic group $\mathbb{Z}_{2^\ell-1}$ is a subgroup of $G_\Gamma^+$;
elements of arbitrarily large order can be obtained as $\ell\rightarrow\infty$.
We complete the proof by showing there are no elements of infinite order in $G_\Gamma^+$. If we diagonalize $\rho_2$ relative to $\rho$,
then the resulting eigenspaces must be preserved by any element $T\in G_\Gamma^+$. The decomposition is given by
$$
\textstyle\left\{\left(f_{1,\mu},\frac{3a_\mu^2+1}{2a_\mu^2}\right),\left(f_{2,\mu},\frac{3a_\mu^2+1}{2a_\mu^2}\right),
\left(f_{3,\mu},\frac{6a_\mu^4+a_\mu^2+4a_\mu^2}{2a_\mu^2}\right)\right\}\,.
$$
For suitable choice of the $a_\mu$, the eigenvalues
$$\textstyle\left\{\frac{3a_\mu^2+1}{2a_\mu^2},\frac{6a_\mu^4+a_\mu^2+4a_\mu^2}{2a_\mu^2}\right\}$$
will all be distinct. Thus $T$ preserves the spaces
$\operatorname{Span}\{f_{1,\mu},f_{2,\mu}\}$ and $\operatorname{Span}\{f_{3,\mu}\}$ individually.
Since $\Gamma(f_{3,\mu},f_{3,\mu},f^{3,\mu})\ne0$, we have $Tf_{3,\mu}=1$. Let $T_\mu$ be the restriction
of $T$ to $\operatorname{Span}\{f_{1,\mu},f_{2,\mu}\}$. Since $T_\mu^2\in\operatorname{{{SO(2)}}}$, we have
$T_\mu^2f_{1,\mu}=\lambda_\mu$ and $T_\mu^2f_{2,\mu}=\bar\lambda_\mu$ for some $\lambda_\mu\in S^1$.
It now follows that $T_\mu^2$ has order at most $2^\ell-1$ so there are no elements of infinite order in $G_\Gamma^+$.
\hfill\qed

\section{The action of $\operatorname{GL}(m,\mathbb{R})$ on $\mathcal{Z}(p,q)$ and on $\mathcal{W}(p,q)$}\label{S4}
Let $\rho_0$ be a symmetric bilinear form of signature $(p,q)$. Let
$$
\operatorname{SO}(\rho_0):=\{T\in\operatorname{GL}^+(m,\mathbb{R}):T^*\rho_0=\rho_0\}\,.
$$
The following is a quite general remark.
\begin{lemma}\label{L4.1} Let $\mathcal{O}$ be a $\operatorname{GL}^+(m,\mathbb{R})$
invariant subset of $\mathcal{Z}(p,q)$ or of $\mathcal{W}(p,q)$. If
action of $\operatorname{SO}(\rho_0)$ on $\mathcal{O}$ is proper, then
the action of $\operatorname{GL}^+(m,\mathbb{R})$ on $\mathcal{O}$ is proper.
\end{lemma}

\begin{proof}  Assume that the action of $\mathcal{SO}(\rho_0)$ on $\mathcal{O}$ is proper.
Suppose given $\Gamma_n\in\mathcal{O}$ and $T_n\in\operatorname{GL}^+(m,\mathbb{R})$ which satisfy
$$
\Gamma_n\rightarrow\Gamma\in\mathcal{O}\text{ and }\tilde\Gamma_n:=T_n\Gamma_n\rightarrow\tilde\Gamma\in\mathcal{O}\,.
$$
We must extract a convergent sequence of the $\{T_n\}$.
Make a change of basis to suppose that $\rho_0=\operatorname{diag}(-1,\dots,-1,+1,\dots,+1)$
relative to the standard basis $\mathcal{B}$ for $\mathbb{R}$. Choose $S\in\operatorname{GL}^+(m,\mathbb{R})$ so that
$S\rho_{\tilde\Gamma}=\rho_0$.  Then $ST_n\Gamma_n\rightarrow S\tilde\Gamma$.
Extracting a convergent subsequence from $ST_n$ is equivalent to extracting a convergent subsequence
from $T_n$. Thus we may assume without loss of generality that $\rho_{\tilde\Gamma}=\rho_0$.
Since $T_n\Gamma_n\rightarrow\tilde\Gamma$,
$\rho_{T_n\Gamma_n}\rightarrow\rho_0$.
We may apply the Gram-Schmidt process to the standard basis $\mathcal{B}$ construct a basis
$\mathcal{B}_n$ for $\mathbb{R}^m$ which is an orthonormal basis for
$\rho_{T_n\Gamma_n}$; since
$\rho_{T_n\Gamma_n}\rightarrow\rho_{\tilde\Gamma}$, the Gram-Schmidt process
does not fail, i.e. we are not trying to normalize a null vector at some stage. Thus the Gram-Schmidt
process yields a sequence $S_n\in\operatorname{GL}^+(m,\mathbb{R})$ so that
$S_n\rightarrow\operatorname{id}$ and so $S_n\rho_{T_n\Gamma_n}=\rho_{S_nT_n\Gamma_n}=\rho_0$. Again,
extracting a convergent subsequence from $S_nT_n$ is equivalent to extracting a convergent
sequence from $T_n$ and hence we may assume without loss of generality that
$\rho_{T_n\Gamma_n}=\rho_0$ for $n$ sufficiently large.
We have
$\Gamma_n=T_n^{-1}\tilde\Gamma_n\rightarrow\Gamma$. Extracting a convergent subsequence from
$\{T_n\}$ is equivalent to extracting a convergent subsequence from $\{T_n^{-1}\}$. Thus
we may interchange the roles of $\{\Gamma_n,\Gamma\}$ and $\{\tilde\Gamma_n,\tilde\Gamma\}$
and apply the argument given above to assume without loss of generality that $\rho_\Gamma=\rho_0$ and
$\rho_{\Gamma_n}=\rho_0$ as well. But since
$\rho_0=\rho_{T_n\Gamma_n}=T_n\rho_{\Gamma_n}=T_n\rho_0$,
$T_n\in\operatorname{SO}(\rho_0)$. By hypothesis, as desired, we can extract a convergent sequence.
\end{proof}

\subsection*{The proof of Theorem~\ref{T1.8}~(1)} Suppose that $(p,q)\in\{(m,0),(0,m)\}$. Then $\rho_0$ is definite
and hence $\operatorname{SO}(\rho_0)$ is compact. Thus any sequence of elements $T_n$ in $\operatorname{SO}(\rho_0)$
has a convergent subsequence so the action of $\operatorname{SO}(\rho_0)$ on $\mathcal{W}(p,q)$ or on $\mathcal{Z}(p,q)$
is proper. Thus by Lemma~\ref{L4.1}, the same is true of the action by $\operatorname{GL}^+(m,\mathbb{R})$.

\subsection*{The proof of Theorem~\ref{T1.8}~(2)} Let
$\Gamma_n\in\mathcal{W}(p,q)$ with $\Gamma_n\rightarrow\Gamma\in\mathcal{W}(p,q)$
and with $G_{\Gamma_n}^+\ne\{\operatorname{id}\}$. We wish to show $G_\Gamma^+\ne\{\operatorname{id}\}$.
This will show that the Christoffel symbols with non-trivial isotropy subgroup form a closed set and correspondingly that
the Christoffel symbols with trivial isotropy subgroup form an open set.

Choose $\operatorname{id}\ne T_n\in G_{\Gamma_n}^+$. Since the action of $\operatorname{GL}^+(m,\mathbb{R})$ is
proper, we can choose a convergent subsequence $T_{n_k}\rightarrow T$. We must guard against the possibility that $T=\operatorname{id}$. Let $\exp$ be the exponential map from the Lie algebra $\mathfrak{so}(\rho_0)$ to
$\operatorname{SO}(\rho_0)$. Put
a Euclidean metric on $\mathfrak{so}(\rho_0)$.
There exists $\varepsilon>0$ so that $\exp$ is a diffeomorphism from the open ball
$B_{3\varepsilon}(0)$ of radius $3\varepsilon$ in $\mathfrak{so}(\rho_0)$ to a neighborhood of
$\operatorname{id}$ in $\operatorname{SO}(\rho_0)$.
Let $\mathcal{O}_k:=\exp(B_{k\varepsilon}(0))$ for $k=1,2$. If $T\in\mathcal{O}_1$ with
$T\ne\operatorname{id}$, then $T=\exp_P(\xi)$ for $0<\|\xi\|<\varepsilon$.
 Choose $k(T)\in\mathbb{N}$
so that $\varepsilon<k(T)\|\xi\|<2\varepsilon$. We then have $T^{k(T)}\in\mathcal{O}_2-\mathcal{O}_1$.
We return to our sequence $\operatorname{id}\ne T_n\in G_{\Gamma_n}\subset\operatorname{SO}(\rho_0)$. By replacing $T_{n_k}$ by
$T_{n_k}^{k(T_n)}$, we can assume additionally that $T_{n_k}\in\mathcal{O}_1^c$ and thus
$T\in\mathcal{O}_{\varepsilon}^c$ so $T\ne\operatorname{id}$.
Since by continuity $T\Gamma=\Gamma$, we conclude as desired that $G_\Gamma^+\ne\operatorname{id}$.\hfill\qed

\subsection*{The proof of Theorem~\ref{T1.8}~(3)} These Assertions follow from Theorem~\ref{T1.5} and
from Assertions~(1,2).
\hfill\qed

\section{The proof Theorem~\ref{T1.9}}\label{S5}

Let $(p,q)$ be given with $m=p+q\ge3$, $p\ge1$, and $q\ge1$. We must show that there exists $\Gamma\in\mathcal{Z}(p,q)$
so that $G_\Gamma^+$ is non-compact. It then follows that the action of $\operatorname{GL}^+(2,\mathbb{R})$ on
$\mathcal{Z}(p,q)$ is not proper.

Suppose first $m=3$. We may work in the torsion free setting.
We consider the structure of Theorem~\ref{T1.12}~(1) and set $\Gamma_{12}{}^3=\Gamma_{21}{}^3=a$,
$\Gamma_{13}{}^1=\Gamma_{31}{}^1=b$, $\Gamma_{23}{}^2=\Gamma_{32}{}^2=c$, $\Gamma_{33}{}^3=d$.  We may then compute that
$$
\rho_\Gamma=\left(\begin{array}{ccc}0&ad&0\\ad&0&0\\0&0&-b^2+bd+c(-c+d)\end{array}\right)\,.
$$
By adjusting the parameters $\{a,b,c,d\}$ suitably, we can obtain either signature $(1,2)$ or signature $(2,1)$.
Let $T_\alpha e_1=\alpha e_1$, $T_\alpha e_2=\alpha ^{-1}e_2$, $T_\alpha e_3=e_3$. This
gives a Lie group isomorphic to $\operatorname{SO}(1,1)$.  We verify that
$T_\alpha\Gamma=\Gamma$ for any $\alpha$.
The sequence $T_n$ obtained by taking $\alpha=n$
then satisfies $\|T_ne_1\|\rightarrow\infty$. Consequently no subsequence of this sequence
converges in $\operatorname{GL}^+(2,\mathbb{R})$.
Therefore, $G_\Gamma^+$ is non compact and the action of $\operatorname{GL}^+(2,\mathbb{R})$ is not proper.

If $m>3$, extend the structure considered above by adding the (possibly) non-zero
Christoffel symbols $\Gamma_{uv}{}^3=\Gamma_{vu}{}^3=\varepsilon_{uv}$ for $4\le u\le v\le m$ where $\varepsilon_{uv}$ are
to be determined. Recall from Equation~(\ref{E1.b}) that
$\rho_{jk}=\Gamma_{in}{}^i\Gamma_{jk}{}^n-\Gamma_{jn}{}^i\Gamma_{ik}{}^n$.
The new Christoffel symbols involving $\varepsilon_{uv}$ make no contribution to $\rho_{jk}$ for indices $1\le j\le k\le 3$ and
only contribute to $\Gamma_{33}{}^3\Gamma_{jk}{}^3$ for $4\le j,k\le m$. Thus
$\rho_{\Gamma,ab}=d\varepsilon_{ab}$ for $4\le a\le b\le m$.
So by adding in these terms, we can obtain any
indefinite signature in dimension at least $3$.\hfill\qed

\section{The two dimensional setting. The proof of Theorem~\ref{T1.11}}\label{S6}
Let $\Gamma_2$ be the structure of Definition~\ref{D1.10}.
Let $\operatorname{SO}(2)=\{T_\theta:0\le\theta<2\pi\}$ and $\operatorname{SO}(1,1)=\{\tilde T_a:a\ne0\}$ where
\begin{equation}\label{E6.a}
T_\theta:=\left(\begin{array}{rr}
\cos(\theta)&\sin(\theta)\\
-\sin(\theta)&\cos(\theta)\end{array}\right)\text{ and }
\tilde T_a:=\left(\begin{array}{ll}
a&0\\0&a^{-1}\end{array}\right)\,.
\end{equation}
Suppose $\Gamma\in\mathcal{Z}(1,1)$ so
the structure group is $\operatorname{SO}(1,1)$.  Then
\begin{equation}\label{E6.b}
\tilde T_a^*\Gamma_{ij}{}^k=a^\varepsilon\text{ for }\varepsilon=\pm1+\pm1+\pm1\in\{\pm1,\pm3\}\,.
\end{equation}
Suppose $\tilde T_a^*\Gamma=\Gamma$. Choose $\Gamma_{ij}{}^k\ne0$. Then
$\tilde T_a^*\Gamma_{ij}{}^k=\Gamma_{ij}{}^k$ implies $a=1$ and $\tilde T_a=\id $. Suppose $\Gamma_n\rightarrow\Gamma$
and $g_n\Gamma_n\rightarrow\tilde\Gamma$. We apply the argument of Lemma~\ref{L4.1} to see that we may suppose
$$
\rho_{s,\Gamma}=\left(\begin{array}{ll}0&1\\1&0\end{array}\right)\text{ and }
g_n=\tilde T_{a_n}\in\operatorname{SO}(1,1)\,.
$$
We use Equation~(\ref{E6.b}) to see that the $\alpha_n$ must converge. Assertion~1 now follows.

Assume $\rho_\Gamma$ is definite so the structure group is $\operatorname{SO}(2)$.
Assume $G_\Gamma$ is non trivial.
Let $\id \ne T_\theta^*\Gamma\in G_\Gamma^+$.
We {{complexify}}. Let $\{e_1,e_2\}$ be the standard basis for $\mathbb{R}^2$. Set
$z:=e_1+\sqrt{-1}e_2$. Then $\{z,\bar z\}$ is a $\mathbb{C}$ basis for
$\mathbb{R}^2\otimes_{\mathbb{R}} {{\mathbb{C}}}$. We extend $\Gamma$ to be complex linear
to define the corresponding complex Christoffel symbols
$\{\Gamma_{zz}{}^z,\ \Gamma_{z\bar z}{}^z,\ \Gamma_{\bar z\bar z}{}^z\}$.
Since the underlying structure is real, the remaining symbols are determined:
$$
\bar\Gamma_{zz}{}^z=\Gamma_{\bar z\bar z}{}^{\bar z},\quad
\bar\Gamma_{z\bar z}{}^z=\Gamma_{z\bar z}{}^{\bar z},\quad
\bar\Gamma_{\bar z\bar z}{}^z=\Gamma_{zz}{}^{\bar z}\,.
$$
Let $ \alpha:=e^{\sqrt{-1}\theta}$.
Since $T_\theta e_1=\cos(\theta) e_1-\sin(\theta) e_2$ and $T_\theta e_2=\sin(\theta) e_1+\cos(\theta) e_1$,
\begin{eqnarray*}
T_\theta z&=&(\cos(\theta)+\sqrt{-1}\sin(\theta))e_1+(-\sin(\theta)+\sqrt{-1}\cos(\theta))e_2\\
&=&(\cos(\theta)+\sqrt{-1}\sin(\theta))(e_1+\sqrt{-1}e_2)=\alpha z\,.
\end{eqnarray*}
Since $T_\theta z=\alpha z$, we have dually that $T_\theta z^*=\bar\alpha z^*$. Consequently when
we raise indices, we have:
$$
(T_\theta ^*\Gamma)_{zz}{}^z=\alpha \alpha \bar \alpha \Gamma_{zz}{}^z,\quad
(T_\theta ^*\Gamma)_{z\bar z}{}^z=\alpha \bar \alpha \bar \alpha \Gamma_{z\bar z}{}^z,\quad
(T_\theta ^*\Gamma)_{\bar z\bar z}{}^z=\bar \alpha \bar \alpha \bar \alpha \Gamma_{\bar z\bar z}{}^z\,.
$$
We have assumed that $T_\theta\ne\id $. If $\Gamma_{zz}{}^z\ne0$, then $\alpha =1$.
Similarly, if $\Gamma_{z\bar z}{}^z\ne0$, then $\bar \alpha =1$. So the only possibility left
is that $\Gamma_{\bar z\bar z}{}^z\ne0$ in which case $\alpha^3=\id $ and
$\theta=\frac{2\pi}3$ or $\theta=\frac{4\pi}3$. This implies $G_\Gamma=\mathbb{Z}_3$.
By making a coordinate rotation and then rescaling, we may assume
$\Gamma_{\bar z\bar z}{}^z=2\sqrt2$. We have:
\begin{eqnarray*}
0&=&\Gamma_{zz}{}^z=\left\{\Gamma_{11}{}^1+2\Gamma_{12}{}^2-\Gamma_{22}{}^1\right\}
+\sqrt{-1}\left\{-\Gamma_{11}{}^2+2\Gamma_{12}{}^1+\Gamma_{22}{}^2\right\},\\
0&=&\Gamma_{z\bar z}{}^z=\left\{\Gamma_{11}{}^1+\Gamma_{22}{}^1\right\}\qquad\hspace{15.5pt}
+\sqrt{-1}\left\{-\Gamma_{11}{}^2-\Gamma_{22}{}^2\right\},\\
2\sqrt2&=&\Gamma_{\bar z\bar z}{}^z=\left\{\Gamma_{11}{}^1-2\Gamma_{12}{}^2-\Gamma_{22}{}^1\right\}
+\sqrt{-1}\left\{-\Gamma_{11}{}^2-2\Gamma_{12}{}^1+\Gamma_{22}{}^2\right\}\,.
\end{eqnarray*}
We solve these equations to obtain the structure $\Gamma_2$.  The remainder of the argument is similar to that given in the
indefinite setting and is therefore omitted.
\qed

\begin{remark}\rm
We normalized the coordinates so that $\Gamma_{\bar z\bar z}{}^z=2\sqrt{2}$. Subsequently, {{in the proof of
Theorem~\ref{T1.12}~(3)}}, we shall deal with the full orbit $\operatorname{GL}^+(2,\mathbb{R})\Gamma_2$ and will not adopt this
normalization.
\end{remark}
\section{The proof of Theorem~\ref{T1.12}}\label{S7}

Let $\mathcal{M}=(\mathbb{R}^3,\Gamma)\in\mathcal{Z}(p,q)$ for $p+q=3$. We suppose $\operatorname{id}\ne T\in G_\Gamma^+$.
In Section~\ref{S7.1}, we will show there is an axis of rotation $\xi$ which is not a null vector so $T\xi=\pm\xi$. It then
follows that $T$ preserves $\xi^\perp$ so the problem becomes, in a certain sense, 2-dimensional.  In Section~\ref{S7.2}, we
show that if $T$ is an element of order at least $4$, then $\dim\{G_\Gamma^+\}\ge1$; this focuses attention on the elements of order $2$
and order $3$ and in Section~\ref{S7.3}, we establish a technical result for elements of order $2$ and $3$.
We use these results in Sections~\ref{S7.4}--\ref{S7.7} to complete the proof of Theorem~\ref{T1.12}.

\subsection{The axis of rotation}\label{S7.1}
We begin our study with the following result:
\begin{lemma}\label{L7.1}
Let $\mathcal{M}=(\mathbb{R}^3,\Gamma)\in\mathcal{Z}(p,q)$ where $p+q=3$.
If $T\in G_\Gamma^+$, then there exists a non-null vector $\xi$ so $T\xi=\xi$.
\end{lemma}

\begin{proof} Let $\rho:=\rho_{s,\Gamma}$ provide a non-degenerate symmetric bilinear inner-product on $\mathbb{R}^3$.
Let $T\in G_\Gamma^+\subset\operatorname{SO}(\rho)$. The characteristic polynomial of $T$ is a
cubic polynomial. Since every cubic polynomial has a real root,
there exists $e_1\ne0$ so $Te_1=ae_1$ for $a\ne0$. Our first task is to show that we can choose an eigenvector which is
not a null vector. Suppose, to the contrary, that $\rho(e_1,e_1)=0$.
Choose $e_3$ so that $\rho(e_1,e_3)=1$. By subtracting
an appropriate multiple of $e_1$, we can assume $\rho(e_3,e_3)=0$. Choose
$e_2\in\operatorname{Span}\{e_1,e_3\}^\perp$. Normalize $e_2$ so
$\rho(e_2,e_2)=\pm1$. Express
$$
Te_1=ae_1,\quad Te_2=t_{21}e_1+t_{22}e_2+t_{23}e_3,\quad
Te_3=t_{31}e_1+t_{32}e_2+t_{33}e_3\,.
$$
As $\rho(Te_1,Te_2)=\rho(e_1,e_2)=0$, $t_{23}=0$.
As $\rho(Te_1,Te_3)=\rho(e_1,e_3)=1$, $t_{33}=a^{-1}$.
Consequently,
$$
Te_1=ae_1,\quad Te_2=t_{21}e_1+t_{22} e_2,\quad Te_3=t_{31}e_1+t_{32}e_2+a^{-1}e_3\,.
$$
We have $\det(T)=t_{22}=1$. Thus $t_{22}=1$. This shows that the eigenvalues of $T$ are $\{1,a,a^{-1}\}$.
Consequently, we could have chosen $e_1$ in the first instance so $Te_1=1$ and we may therefore assume $a=1$
and consequently $a$ is the only eigenvalue of $T$. This implies that
$$Te_1=e_1,\quad Te_2=t_{21}e_1+e_2,\quad Te_3=t_{31}e_1+t_{32}e_1+e_3\,.
$$
If $t_{21}=0$, then we may take $\xi=e_2$ to establish the desired result.
Thus
$t_{21}=b\ne0$. Because
$\rho(Te_2,Te_3)=0$ and $\rho(Te_3,Te_3)=0$,
$$
Te_1=e_1,\quad
Te_2=be_1+e_2,\quad
Te_3=\textstyle\frac12b^2\rho(e_2,e_2)e_1-be_2+e_3\text{ for }b\ne0\,.
$$
Note that $T^*\Gamma$ is a polynomial in $b$. If we replace $T$ by $T^n$,
we replace $b$ by $nb$. Thus if $\Gamma_{ij}{}^k\ne0$,
all the coefficients of $b^k$ must vanish in $(T^*\Gamma)_{ij}{}^k$. We linearize
the problem and work modulo terms which are quadratic and of higher order in $b$ and
concentrate on the relations provided by the linear terms. Let $\{e^1,e^2,e^3\}$ be the
dual basis. Expand:
$$\begin{array}{lll}
Te_1\equiv e_1,\qquad\phantom{......} Te_2\equiv be_1+e_2,&Te_3\equiv-be_2+e_3,\\[0.05in]
Te^1\equiv e^1-be^2,\quad Te^2\equiv e^2+be^3,&Te^3\equiv e^3,\\[0.05in]
T^*\Gamma_{12}{}^3\equiv b\Gamma_{11}{}^3+\Gamma_{12}{}^3,&
T^*\Gamma_{13}{}^3\equiv-b\Gamma_{12}{}^3+\Gamma_{13}{}^3,\\[0.05in]
T^*\Gamma_{23}{}^3\equiv b\Gamma_{13}{}^3-b\Gamma_{22}{}^3+\Gamma_{23}{}^3,&
T^*\Gamma_{33}{}^3\equiv-2b\Gamma_{23}{}^3+\Gamma_{33}{}^3.
\end{array}$$
We set the terms involving $b$ to zero to see:
$$\begin{array}{lllll}
\Gamma_{11}{}^3=0,&\Gamma_{12}{}^3=0,&\Gamma_{13}{}^3=c_1,&\Gamma_{22}{}^3=c_1,&
\Gamma_{23}{}^3=0.
\end{array}$$
We continue the expansion
$$\begin{array}{ll}
T^*\Gamma_{12}{}^2\equiv b\Gamma_{11}{}^2+b\Gamma_{12}{}^3+\Gamma_{12}{}^2,&
T^*\Gamma_{13}{}^2\equiv b\Gamma_{13}{}^3-b\Gamma_{12}{}^2+\Gamma_{13}{}^2,\\ [0.05in]
T^*\Gamma_{22}{}^2\equiv 2b\Gamma_{12}{}^2+b\Gamma_{22}{}^3+\Gamma_{22}{}^2,&
T^*\Gamma_{23}{}^2\equiv b\Gamma_{13}{}^2-b\Gamma_{22}{}^2+b\Gamma_{23}{}^3+\Gamma_{23}{}^2,\\[0.05in]
T^*\Gamma_{33}{}^2\equiv -2b\Gamma_{23}{}^2+b\Gamma_{33}{}^3+\Gamma_{33}{}^2.
\end{array}$$
We set the terms involving $b$ to zero. We use the previous relations to $c_1=0$ and
$$\begin{array}{llllll}
\Gamma_{11}{}^2=0,&\Gamma_{12}{}^2=0,&\Gamma_{22}{}^2=c_2,&\Gamma_{13}{}^2=c_2,&\Gamma_{23}{}^2=c_3,\\[0.05in]
\Gamma_{11}{}^3=0,&\Gamma_{12}{}^3=0,&\Gamma_{13}{}^3=0,&
\Gamma_{22}{}^3=0,&
\Gamma_{23}{}^3=0,\\[0.05in]
\Gamma_{33}{}^3=2c_3.
\end{array}$$
We continue the computation:
$$\begin{array}{ll}
T^*\Gamma_{12}{}^1\equiv b\Gamma_{11}{}^1+\Gamma_{12}{}^1,\\ [0.05in]
T^*\Gamma_{13}{}^1\equiv -b\Gamma_{12}{}^1-b\Gamma_{13}{}^2+\Gamma_{13}{}^1,&
T^*\Gamma_{22}{}^1\equiv 2b\Gamma_{12}{}^1-b\Gamma_{22}{}^2+\Gamma_{22}{}^1,\\ [0.05in]
T^*\Gamma_{23}{}^1\equiv b\Gamma_{13}{}^1-b\Gamma_{22}{}^1-b\Gamma_{23}{}^2+\Gamma_{23}{}^1,&
T^*\Gamma_{33}{}^1\equiv -2b\Gamma_{23}{}^1-b\Gamma_{33}{}^2+\Gamma_{33}{}^1.
\end{array}$$
We set the terms involving $b$ to zero and use the previous relations to see $c_2=0$ and obtain:
$$\begin{array}{llllll}
\Gamma_{11}{}^1=0,&\Gamma_{12}{}^1=0,&\Gamma_{22}{}^1=c_5,&\Gamma_{13}{}^1=c_3+c_5,&
\Gamma_{23}{}^1=c_4,\\[0.05in]
\Gamma_{11}{}^2=0,&\Gamma_{12}{}^2=0,&\Gamma_{22}{}^2=0,&
\Gamma_{13}{}^2=0,&\Gamma_{23}{}^2=c_3,\\[0.05in]
\Gamma_{11}{}^3=0,&\Gamma_{12}{}^3=0,&\Gamma_{22}{}^3=0,&\Gamma_{13}{}^3=0,&
\Gamma_{23}{}^3=0,\\[0.05in]
\Gamma_{33}{}^1=c_6,&\Gamma_{33}{}^2=-2c_4,&\Gamma_{33}{}^3=2c_3.
\end{array}$$
By Equation~(\ref{E1.b}), $\rho_{jk}=\Gamma_{in}{}^i\Gamma_{jk}{}^n-\Gamma_{jn}{}^i\Gamma_{ik}{}^n$.
Consequently
$$\rho_{k1}=\rho_{1k}=\Gamma_{in}{}^i\Gamma_{1k}{}^n-\Gamma_{1n}{}^i\Gamma_{ik}{}^n
=\delta_{k,3}\Gamma_{i1}{}^i\Gamma_{13}{}^1-\Gamma_{13}{}^1\Gamma_{1k}{}^3=0-0=0\,.
$$
This shows that $\rho$ is singular which is false.

Thus we can choose an eigenvector which is not a null vector. Choose $\xi$ so $T\xi=a\xi$ and $\rho(\xi,\xi)\ne0$.
Since $\rho(\xi,\xi)=\rho(T\xi,T\xi)=a^2\rho(\xi,\xi)$, we conclude $a^2=1$. Suppose that $a\ne1$ so $a=-1$.
Let $V=\xi^\perp$. Then $\rho_V:=\rho|_V$ is non-degenerate. Furthermore, $T\xi^\perp=\xi^\perp$. Let $T_V:=T|_V$.
Since $T\xi=-\xi$ and $\det(T)=1$, $\det(T_V)=-1$. The characteristic polynomial $T_V$ takes the form
$\lambda^2+c_1\lambda+\det(T_V)=0$. Since $\det(T_V)=-1$, there are 2 real eigenvalues of $T_V$ and $T_V$ is
diagonal. Thus we can choose a basis $\{\eta,\sigma\}$ for $V$ so $T\eta=\lambda_1\eta$ and $T\sigma=\lambda_2\sigma$
where $\lambda_1\lambda_2=-1$.
Since $\rho(\eta\sigma)=\rho(T\eta,T\sigma)=\lambda_1\lambda_2\rho(\eta,\sigma)=-\rho(\eta,\sigma)$, we have $\rho(\eta,\sigma)=0$.
Since $\rho_V$ is non-degenerate, neither $\eta$ nor $\sigma$ is null. Thus $\lambda_1^2=\lambda_2^2=1$. Since $\lambda_1\lambda_2=-1$, there must exist a $+1$ eigenvector of $T$ which is not null.
\end{proof}

\subsection{Elements of order at least 4}\label{S7.2}
The following subgroups of $\operatorname{GL}^+(3,\mathbb{R})$
will play a central role in what follows. {We generalize Equation~(\ref{E6.a}) to the 3 dimensional setting to define:}
\begin{eqnarray*}
&&\operatorname{SO}(2):=\left\{
T_\theta:=\left(\begin{array}{rrr}
\cos(\theta)&\sin(\theta)&0\\
-\sin(\theta)&\cos(\theta)&0\\
0&0&1
\end{array}\right)\text{ for }0\le\theta\le2\pi\right\}\,,\\
&&\operatorname{SO}(1,1):=\left\{
\tilde T_a:=\left(\begin{array}{lll}
a&0&0\\0&a^{-1}&0\\0&0&1
\end{array}\right)\text{ for  }a\in\mathbb{R}-\{0\}\right\}\,.
\end{eqnarray*}

\begin{lemma}\label{L7.2}
Let $T\in G_\Gamma^+$ for $\Gamma\in\mathcal{Z}(p,q)$ with $p+q=3$. If $T$ has order at least $4$, then
$\dim\{G_\Gamma^+\}=1$ and after making a suitable choice of basis we have either that
$\operatorname{SO}(2)\subset G_\Gamma^+$ or that $\operatorname{SO}(1,1)\subset G_\Gamma^+$.
\end{lemma}

\begin{proof} Choose a unit vector $e_3$ so $Te_3=e_3$. Let $V=e_3^\perp$ and let $\rho_V:=\rho|_V$; $\rho_V$ is non-degenerate.
Furthermore, $T$ preserves $V$. Let $T_V:=T|_V$. Since $\det(T)=1$ and $Te_3=e_3$, $T_V\in\operatorname{SO}(\rho_V)$.
{ We apply the same argument as that used to establish Lemma~\ref{T1.11}.}
\smallbreak\noindent{\bf Case 1.} Suppose $\rho_V$ is indefinite.
Choose a hyperbolic basis $\{e^1,e^2\}$ for $V$ so $\rho_V=e^1\otimes e^2+e^2\otimes e_1$. Since $T_V\in\operatorname{SO}(\rho_V)$,
there exists $a$ so $T=T_a$ takes the form $Te_1=ae_1$ and $Te_2=a^{-1}e_2$. Since $T$ has order at least 4, $a\ne 1$. We compute
$$
T\Gamma_{ij}{}^k=a^{\epsilon(ijk)}\Gamma_{ij}{}^k\text{ for }\epsilon(ijk)=\delta_{1i}-\delta_{2i}+\delta_{1j}-\delta_{1k}+\delta_{2k}\,.
$$
Since $\epsilon(ijk)\in(0,\pm1,\pm2,\pm3)$ and $a\ne\pm1$, we conclude $\Gamma_{ij}{}^k=0$ for $\epsilon(ijk)\ne0$.
But this implies that $T_b\in G_\Gamma^+$ for any $b$ and hence $\operatorname{SO}(1,1)\subset G_\Gamma^+$.
\smallbreak\noindent{\bf Case 2.} Suppose that $\rho_V$ is indefinite. We complexify and set
$$\begin{array}{lll}
f_1:=e_1+\sqrt{-1}e_2,&f_2:=e_1-\sqrt{-1}e_2,&f_3:=e_3,\\
f^1:=\frac12(e^1-\sqrt{-1}e^2),&f^2:=\frac12(e^1+\sqrt{-1}e^2),&f^3:=e^3.
\end{array}$$
Since $T\in\operatorname{SO}(2)$, the complex eigenvalues of $T$ are $\{\alpha,\bar\alpha=\alpha^{-1}\}$.
Since we can diagonalize $T$ over $\mathbb{C}$, we may choose the notation so $Tf_1=\alpha f_1$ and $Tf_2=\alpha^{-1}f_2$.
The analysis of Case 1 pertains and thus $T\Gamma_{ij}{}^k=\alpha^{\epsilon(ijk)}\Gamma_{ij}{}^k$. By assumption, $\alpha\ne1$,
$\alpha^2\ne1$, and $\alpha^3\ne1$. Thus once again $\Gamma_{ij}{}^k=0$ if $\epsilon(ijk)\ne0$ and
we may conclude $\operatorname{SO}(2)\subset G_\Gamma^+$.
\end{proof}
\subsection{Elements of order 2 and of order 3}\label{S7.3}
Lemma~\ref{L7.2} focuses attention on the elements of order $2$ and of order $3$. The following is a useful
technical result.

\begin{lemma}\label{L7.3}
$T_i\in G_\Gamma^+$ for $i=1,2$. If
$T_1$, $T_2$, and $T_1T_2$ have order $3$, then either $T_1T_2^2=\id$ or $T_1T_2^2$ has
order $2$.
\end{lemma}

\begin{proof}
Suppose $T\in\operatorname{GL}(3,\mathbb{R})$ can be written in the form
\begin{equation}\label{E7.a}
T=\left(\begin{array}{rrr}\cos(\theta)&\sin(\theta)&0\\-\sin(\theta)&\cos(\theta)&0\\0&0&1\end{array}\right)\,.
\end{equation}
Then $\operatorname{Tr}(T)=2\cos(\theta)+1$ so $\operatorname{Tr}(T)$ determines $T$ up to conjugacy
in this setting:
\begin{enumerate}
\item $T$ has order $1$ $\Leftrightarrow$ $\cos(\theta)=+1$ $\Leftrightarrow$
$\operatorname{Tr}(T)=+3$.
\item $T$ has order $2$ $\Leftrightarrow$ $\cos(\theta)=-1$ $\Leftrightarrow$
$\operatorname{Tr}(T)=-1$.
\item $T$ has order $3$ $\Leftrightarrow$ $\cos(\theta)=-\frac12$ $\Leftrightarrow$
$\operatorname{Tr}(T)=\phantom{-}0$.
\end{enumerate}
Since $T_1\in G_\Gamma^+$ has order $3$, we may choose a non-null vector so $Te_3=\pm e_3$.
Since $T^3=\id$, $Te_3=e_3$. Consequently, $T_1$ has the form given in Equation~(\ref{E7.a}). Let
$P_1:=e_3^\perp$ be the rotation plane of $T_1$ and, similarly, let $P_2$ be the rotation plane
of $T_2$. Since $\dim\{P_1\cap P_2\}\ge\dim\{P_1\}+\dim\{P_2\}-3=1$, we can choose a unit
vector $e_1\in P_1\cap P_2$. Let $\{e_2,f_2\}$ be unit vectors so
$T_1e_1=-\frac12e_1+\frac{\sqrt3}2e_2$ and $T_2{{e_1}}=-\frac12e_1+\frac{\sqrt3}2f_2$.
\smallbreak\noindent{\bf Case 1.}
Assume $\rho$ is definite. Decompose $f_2=xe_2+ye_3$ where
$x^2+y^2=1$. Let $f_3:=-ye_2+xe_3$ be a unit vector which spans the
rotation axis of $T_2$. We then have: \smallbreak\hglue .5cm
$\begin{array}{lll}
f_1=e_1,&f_2=xe_2+ye_3,&f_3=-ye_2+xe_3,\\[0.05in]
e_1=f_1,&e_2=xf_2-yf_3,&e_3=yf_2+xf_3,\\[0.05in]
T_1e_1=-\frac12e_1+\frac{\sqrt3}2e_2,&
T_1e_2=-\frac{\sqrt3}2e_1-\frac12e_2,&T_1e_3=e_3,\\[0.05in]
T_2e_1=-\frac12e_1+\frac{\sqrt3}2f_2,&
T_2f_2=-\frac{\sqrt3}2e_1-\frac12f_2,&T_2f_3=f_3,\\[0.05in]\end{array}$
\smallbreak\qquad
$T_1T_2e_1=T_1\{-\frac12e_1+\frac{\sqrt3}2(xe_2+ye_3)\}
=\frac14(1-3x)e_1+\star e_2+\star e_3$, \smallbreak\qquad
$T_1T_2e_2=T_1T_2(xf_2-yf_3)=T_1(-\frac{x\sqrt3}2e_1-\frac
x2f_2-yf_3)$ \smallbreak\qquad\quad
$=T_1(-\frac{x\sqrt3}2e_1{{+}} (-\frac{x^2}2+y^2)e_2+\star e_3)$
\smallbreak\qquad\quad $=\star e_1+\frac14(-3x+x^2-2y^2)e_2+\star
e_3$, \smallbreak\qquad $T_1T_2e_3=T_1T_2(yf_2+xf_3)=T_1(\star
e_1-\frac12yf_2+xf_3)$ \smallbreak\qquad\quad $=T_1(\star e_1+\star
e_2+(-\frac12y^2+x^2)e_3)$ \smallbreak\qquad\quad $=\star e_1+\star
e_2+\frac14(4x^2-2y^2)e_3$, \smallbreak\qquad
$\operatorname{Tr}(T_1T_2)=\frac14(1-3x-3x+x^2+4
x^2-4y^2)=\frac34({{-1-2x+3x^2}})$. \medbreak\noindent Since
$T_1T_2$ has order $3$, $\operatorname{Tr}(T_1T_2)=0$. The equation
${{-1-2x+3x^2 = 0}}$ then implies $x=1$ or $x=-\frac13$. Since $T_2$
has order $3$, $T_2^2$ also has order $3$. Introduce similar
notation $\{\tilde x,\tilde y,\tilde f_2,\tilde f_3\}$ for $T_2^2$
to expand $\operatorname{Tr}(T_1T_2^2)=\frac34({{-1-2\tilde
x+3\tilde x^2}})$. Because $T_2^2e_1=-\frac12e_1-\frac{\sqrt3}2f_2$,
we see that $\tilde f_2=-f_2$ and thus $\tilde x=-x$.  We complete
proof if $\rho$ is definite by computing:
$$
\operatorname{Tr}(T_1T_2^2)=\textstyle\frac34({{-1+2x+3x^2}})=\left\{\begin{array}{rrr}
3&\text{if}&x=1\\
-1&\text{if}&x=-\frac13\end{array}\right\}\,.
$$
\smallbreak\noindent{\bf Case 2. } Assume $\rho$ is indefinite.
In the hyperbolic setting, we have $x^2-y^2=1$, but the same argument pertains. We compute:
\medbreak\hglue .5cm
$\begin{array}{lll}
T_1e_1=-\frac12e_1+\frac{\sqrt3}2e_2,&
T_1e_2=-\frac{\sqrt3}2e_1-\frac12e_2,&T_1e_3=e_3,\\[0.05in]
T_2e_1=-\frac12e_1+\frac{\sqrt3}2f_2,&
T_2f_2=-\frac{\sqrt3}2e_1-\frac12f_2,&T_2f_3=f_3,\\[0.05in]
f_1=e_1,&f_2=xe_2+ye_3,&f_3=ye_2+xe_3,\\[0.05in]
e_1=f_1,&e_2=xf_2-yf_3,&e_3=-yf_2+xf_3.
\end{array}$
\smallbreak\qquad
$T_1T_2e_1=T_1\{(-\frac12e_1+\frac{\sqrt3}2(xe_2+ye_3)\}
=\frac14(1-3x)e_1+\star e_2+\star e_3$, \smallbreak\qquad
$T_1T_2e_2=T_1T_2(xf_2-yf_3)=T_1(-\frac{x\sqrt3}2f_1-\frac
x2f_2-yf_3)$ \smallbreak\qquad\quad
$=T_1(-\frac{x\sqrt3}2e_1{{+}}(-\frac{x^2}2-y^2)e_2+\star e_3)$
\smallbreak\qquad\quad $=\star
e_1+(-\frac{3x}4+\frac{x^2}4+\frac12y^2)e_2+\star e_3$,
\smallbreak\qquad $T_1T_2e_3=T_1T_2(-yf_2+xf_3)=T_1(\star
f_1+\frac12yf_2+xf_3)$ \smallbreak\qquad\quad $=T_1(\star e_1+\star
e_2+(\frac12y^2+x^2)e_3)$ \smallbreak\qquad\quad $=\star e_1+\star
e_2+(x^2+\frac12y^2)e_3$, \smallbreak\qquad
$\operatorname{Tr}(T_1T_2)=\frac14(1-3x-3x+x^2+2y^2+4x+2y^2)=\frac34(-1-2x+3x^2)$,
\medbreak\noindent The remainder of the argument is the same as in
the definite setting. Assertion~2 now follows.
\end{proof}

Let $\nu(T)$ be the order of $T\in\operatorname{GL}(2,\mathbb{R})$. In what follows we list the (possibly)
non-zero Christoffel symbols up to the symmetry $\Gamma_{ij}{}^k=\Gamma_{ji}{}^k$. Let
\begin{equation}\label{E7.b}
S_je_i=\left\{\begin{array}{rc}e_i&\text{ if }i=j\\-e_i&\text{ if }i\ne j\end{array}\right\}\,.
\end{equation}
 Let $s_3$ be the symmetric
group of all permutations on 3 elements; $s_3$ is a non-abelian group of order 6. Let $a_4$ be
the alternating group of permutations of 4 elements; $a_4$ is a non-Abelian group of order 12.
We assume there exists $T\in G_\Gamma^+$ of order at least 3 as otherwise
Assertion~4 holds of Theorem~\ref{T1.12} holds. By Assertion~1 of Lemma~\ref{L7.1}, we may choose a non-null vector $e_3$ so that
$Te_3=\pm e_3$ and $\rho(e_3,e_3)=\pm1$.
Let $V:=e_3^\perp$, let $\rho_V:=\rho|_V$, and let $T_V:=T|_V$. Then $\rho_V$ is non-degenerate
and $T_V\in\operatorname{SO}(\rho|_V)$. We divide the proof of Theorem~\ref{T1.12} into 4 cases.
\subsection{The proof of Theorem~\ref{T1.12}~(1)}\label{S7.4}
Assume $\rho_V$ is indefinite. Choose a hyperbolic basis for $V$
so
$$
\rho_V(e_1,e_1)=\rho_V(e_2,e_2)=0\text{ and }\rho_V(e_1,e_2)=1\,.
$$
Suppose first $Te_3=e_3$.
Since $\det(T_V)=+1$,
$T_Ve_1=\alpha e_1$ and $T_Ve_2=\alpha^{-1}e_2$.
Since $\nu(T)\ge3$, $\alpha\ne\pm1$.
Since $T^n$ preserves $\Gamma$, the only possible non-zero Christoffel symbols are
$\{\Gamma_{13}{}^1,\Gamma_{23}{}^2,\Gamma_{12}{}^3,\Gamma_{33}{}^3\}$.
Consequently, $\Gamma$ has the form given in Assertion~1 and
the Ricci tensor is as given;
to ensure $\rho$ is non-degenerate, $(a,b,c,d)$ satisfy the given constraints.
We adopt the notation of Equation~(\ref{E7.b}) to define $S_2$.
Then $S_2^*\Gamma=\Gamma$ implies $d=0$ which is false. Thus, in particular,
$S_2\notin G_\Gamma^+$ so $G_\Gamma^+\ne\operatorname{SO}(\rho)$.

Suppose $S\in G_\Gamma^+-\operatorname{SO}(1,1)$.
Since any two distinct connected $1$-dimensional subgroups
generate $\operatorname{SO}(\rho)$ and since
$G_\Gamma^+\ne\operatorname{SO}(\rho)$, $S$ must normalize
$\operatorname{SO}(1,1)$ and in particular preserves $V$ and $e_3$.
Since $S\notin\operatorname{SO}(1,1)$,
$Se_3=-e_3$. But $S^*\Gamma_{33}{}^3=-\Gamma_{33}{}^3$ and hence $d=0$.
This is not possible. Thus
$G_\Gamma^+=\operatorname{SO}(1,1)$.

Next suppose $Te_3=-e_3$. Since $T_V\in\operatorname{O}(\rho_V)$
and $\det(T_V)=-1$,
$Te_1=\alpha e_2$ and $Te_2=\alpha^{-1}e_1$ for some $\alpha\ne0$.
We then have $T^2=\id$ which is false
as we assumed that $\nu(T)\ge3$. This completes the analysis of the case when $\rho_V$ is
indefinite.

\subsection{The proof of Theorem~\ref{T1.12}~(2)}\label{S7.5}
 Assume that $\rho_V$ is definite and that $T$ has order
at least 4.
If $Te_3=-e_3$, then $T_V\in\operatorname{O}(2)-\operatorname{SO}(2)$ and $T^2=\id$
which is false. Thus $Te_3=e_3$ and $T_V\in\operatorname{SO}(2)$.
Let $\{e_1,e_2\}$ be an orthonormal basis for $V$. Since
$T_V\in\operatorname{SO}(V)$, $T_V$ is a rotation through an angle $\theta$ on $e_3^\perp$.
We use the argument used to prove {{Theorem ~\ref{T1.11}}}. Set $f_1=e_1+\sqrt{-1}e_2$, $f_2=e_1-\sqrt{-1}e_2$, and $f_3=e_3$.
We then have $Tf_1=e^{\sqrt{-1}\theta}f_1$, $Tf_2=e^{-\sqrt{-1}\theta}f_2$, and $Tf_3=f_3$.
Let $\tilde\Gamma_{ij}{}^k$ be the complex Christoffel symbols relative to the basis $\{f_1,f_2,f_3\}$.
We have:
$$
T^*\tilde\Gamma_{ij}{}^k=e^{\sigma_{ij}{}^k\sqrt{-1}\theta}\tilde\Gamma_{ij}{}^k\text{ for }
\sigma_{ij}{}^k:=\{\delta_{1i}-\delta_{2i}\}+\{\delta_{1j}-\delta_{2j}\}+\{\delta_{2k}-\delta_{1k}\}\,.
$$
Consequently, $\sigma_{ij}{}^k\in\{-3,-2,-1,0,1,2,3\}$. Since $T_V$ does not
have order $1$, $2$ or $3$, $e^{\sigma_{ij}{}^k\sqrt{-1}\theta}\ne1$ for $\sigma_{ij}{}^k\ne0$ and thus
$\tilde\Gamma_{ij}{}^k=0$ for $\sigma_{ij}{}^k\ne0$.
The possible elements with $\sigma_{ij}{}^k=0$ are
$\{\tilde\Gamma_{12}{}^3,\tilde\Gamma_{13}{}^1,\tilde\Gamma_{23}{}^2,\tilde\Gamma_{33}{}^3\}$. After disentangling
the notation, we conclude that
$\Gamma=\Gamma_{\operatorname{SO}(2)}(a,b,c,d)$
has the form where the (potentially) non-zero entries are given in Assertion~2 of Theorem~\ref{T1.12}
and, consequently,
$\operatorname{SO}(2)\subset\Gamma_{\operatorname{SO}(2)}(a,b,c,d)$. One
then computes the Ricci tensor and obtains the required conditions on $(a,b,c,d)$.

We complete the proof of Assertion~2 by showing $\operatorname{SO}(2)=G_\Gamma^+$.
Adopt the notation of Equation~(\ref{E7.b}) to define $S_2$.
If $S_2^*\Gamma=\Gamma$, then
$a=0$ which is false. Thus
$G_\Gamma^+\ne\operatorname{SO}(\rho)$.
Suppose there exists $S\in G_\Gamma^+-\operatorname{SO}(2)$. As any 2 distinct connected
1-dimensional Lie subgroups of $\operatorname{SO}(\rho)$ generate $\operatorname{SO}(\rho)$
and as $G_\Gamma^+\ne\operatorname{SO}(\rho)$, $S$ must normalize $\operatorname{SO}(2)$ so,
in particular, $S$ preserves $V$ and, consequently, $Se_3=\pm e_3$. If $Se_3=e_3$, then
$S\in\operatorname{SO}(2)$ which is false. Thus $Se_3=-e_3$. Since $\det(S_V)=-1$, $S$
fixes a vector of $V$. Choose the basis so $Se_2=e_2$. It then
follows $Se_1=-e_1$ so $S=S_2\in G_\Gamma^+$ which contradicts the argument we have just
given. Thus $G_\Gamma^+=\operatorname{SO}(2)$ and Assertion 2 holds.

\subsection{The proof of Theorem~\ref{T1.12}~(3)}\label{S7.6}
If there exists an element $T\in G_\Gamma^+$ of order
at least $4$, then the analysis given above to examine Assertion~1 or Assertion~2 pertains and
$G_\Gamma^+=\operatorname{SO}(1,1)$ or $G_\Gamma^+=\operatorname{SO}(2)$. Thus
we conclude that the order of any element $T\in G_\Gamma^+$ is at most 3. Furthermore,
if $T$ has order $3$ and if $Te_3=\pm e_3$, then $\rho_V$ is definite. Finally, since Assertion~4
does not hold, there exists an element of order $3$. Fix such an element. In the proof of Assertion~3,
we ignored the case where
$\sigma_{ij}{}^k=\pm3$. When we include these cases, we must allow $\tilde\Gamma_{11}{}^2$
and the complex conjugate $\tilde\Gamma_{22}{}^1$. This
shows $\Gamma=e\Gamma_2+\Gamma_{\operatorname{SO}}(a,b,c,d)$ is as described in
Assertion~2. If $e=0$, then the analysis of Assertion~2
pertains and $\Gamma_G^+=\operatorname{SO}(2)$.
Thus we may assume $e\ne0$.

In Assertion~3, $G_\Gamma^+$ can be bigger than $\mathbb{Z}_3$ in certain instances
and we must examine these exceptional structures. We
suppose there exists $S\in G_\Gamma^+-\{1,T,T^2\}$.
We wish to show $S$ can be chosen so $S$ has order $2$. Suppose to the contrary that $S$
has order $3$. If $TS$ has order $2$, we have an element of order $2$. So we assume $TS$ has
order $3$. But the Assertion~2 of Lemma~\ref{L7.3} shows that $TS^2=\id$ or $TS^2$ has order $2$.
Since $S\notin\{1,T,T^2\}$, $TS^2\ne \id$.

 We use the argument used to prove the second assertion of {{Theorem~\ref{T1.11}}}. We had the exceptional structure $\Gamma_2$.
We made a coordinate rotation to ensure $\Gamma_{zz}{}^{\bar z}=1$.
If instead, we assume that $\Gamma_{zz}{}^{\bar z}=4e+4\sqrt{-1}f$, then we obtain
a slightly more general form
$$\begin{array}{llllll}
\Gamma_{11}{}^1=e,&\Gamma_{11}{}^2=-f,&\Gamma_{11}{}^3=a,&
\Gamma_{12}{}^1=-f,&\Gamma_{12}{}^2=-e,&\Gamma_{12}{}^3=0,\\ [0.05in]
\Gamma_{13}{}^1=b,&\Gamma_{13}{}^2=c,&\Gamma_{13}{}^3=0,&
\Gamma_{22}{}^1=-e,&\Gamma_{22}{}^2=f,&\Gamma_{22}{}^3=a,\\ [0.05in]
\Gamma_{23}{}^1=-c,&\Gamma_{23}{}^2=b,&\Gamma_{23}{}^3=0,&
\Gamma_{33}{}^1=0,&\Gamma_{33}{}^2=0,&\Gamma_{33}{}^3=d.
\end{array}$$
Let $S$ be an element of order $2$ in $G_\Gamma^+$.
Let $N(S,-1)$ be the $-1$ eigenspace of $S$ and let $V=e_3^\perp$. Since
$\dim\{N(S,-1)\}+\dim\{V\}-3=1$, $N(S,-1)$ intersects $V$. If $N(S,-1)=V$, then
$S$ commutes with $T$ so $TS$ is an element of order 6 which is impossible.
Thus $N(S,-1)\cap V$ is 1-dimensional. We choose the basis for $V$ so
$N(S,-1)\cap V=e_2\cdot\mathbb{R}$; this implies $f=0$.
We have $N(TST^{-1},-1)\cap V=Te_2\cdot\mathbb{R}$ and $N(T^2ST^{-2},-1)\cap V=T^2e_2\cdot\mathbb{R}$.
Thus $\{S,TST^{-1},T^2ST^{-2}\}$ are 3 distinct elements of order $2$ in $G_\Gamma^+$.
We have already ruled out the case $Se_3=e_3$. There are 3 remaining cases:

\smallbreak\noindent{\bf Case 1.} Suppose $Se_3=-e_3$. This means
$S\in\operatorname{O}(2)-\operatorname{SO}(2)$.
So we can choose the basis so $Se_1=e_1$ and $Se_2=-e_2$).
This implies $f=0$, $a=0$, $b=0$, and $d=0$. We obtain
$\rho=\operatorname{diag}(-2e^2,-2e^2,2c^2)$. In particular, $\rho$ is indefinite.
We can renormalize the coordinates so $e=1$ and $c=1$
to obtain the structure of Assertion~1. We have $STS^{-1}=T^{-1}$; $\{\id,T,T^2,S,ST,ST^2\}$
is a non-Abelian group of order 3 and
hence isomorphic to $s_3$. We note that $\{T,T^2\}$ are the elements of order $3$ in $s_3$ and
$\{S,TST^{-1},T^2ST^{-2}\}$ are the elements of order $2$ in $s_3$.

Suppose that $\tilde S$ is an element of order $2$ in $G_\Gamma^+$  which does not belong to
$s_3$. Since we have established the signature is indefinite, the argument we will give
in Case 2 below shows that
$\tilde Se_3=-e_3$ as well. Since $S\tilde Se_3=e_3$,
$S\tilde S$ belongs to $\operatorname{SO}(2)\cap G_\Gamma^+=\{\id,T,T^2\}$ and $\tilde S\in s_3$.
Thus the elements of order $2$ in $G_\Gamma^+$ are given by $\{S,TST^{-1},T^2ST^{-2}\}$.
Suppose $\tilde T$  is an element of $G_\Gamma^+$ which is not in $s_3$. Thus $\tilde T$ has
order $3$. Since $\tilde TT$ is not in $s_3$, $\tilde TT$ has order 3 as well. Lemma~\ref{L7.3}
then implies $\tilde TT^2$ has order $1$ or order $2$ and hence belongs to $s_3$. This implies
$\tilde T$ belongs to $s_3$. This contradiction then shows, as desired,that $G_\Gamma^+=s_3$.

\smallbreak\noindent{\bf Case 2. } Suppose $Se_3\ne\pm e_3$ and $\rho$ is indefinite.
We rescale $e_1$ to assume $e=1$. We rescale $e_2$ and $e_3$ to assume
$$
\rho(e_1,e_1)=\rho(e_2,e_2)=-\rho(e_3,e_3)\ne0\text{ and }\rho(e_i,e_j)=0\text{ for }i\ne j\,.
$$
We find $\{x,y\}$ so $x^2-y^2=1$ with $y\ne0$. It then follows $x\ne0$. Express:
$$\begin{array}{lll}
Se_1=xe_1+ye_3,&Se_2=-e_2,&Se_3=-ye_1-xe_3,\\ [0.05in]
Se^1=xe^1-ye^3,&Se^2=-e^2,&Se^3=ye^1-xe^3.
\end{array}$$
Let $\Delta_{ij}{}^k:=(S_*\Gamma)_{ij}{}^k-\Gamma_{ij}{}^k$. We have
$0=\Delta_{22}{}^2=-2f$.  Thus $f=0$. We then have
$$
0=\Delta_{11}{}^2=-2cxy,\quad
0=\Delta_{22}{}^1=1-x-ay,\quad
0=\Delta_{12}{}^2=1-x+by\,.
$$
Since $xy\ne0$, we have $c=0$, $b=-a$, and $x=1-ay$.
We impose these relations and compute:
$$
0=\Delta_{23}{}^2=2a+(1-a^2)y\,.$$
This implies $a^2\ne1$. We obtain therefore
$$y=\frac{2a}{a^2-1}\quad\text{and}\quad x=\frac{1+a^2}{1-a^2}\,.$$
We verify $x^2-y^2=1$. Since $y\ne0$, $a\ne0$. The relations $\Delta_{11}{}^1=0$ and $\Delta_{33}{}^3=0$
imply
$$
3+3a^2+2a^4+2ad=0,\quad 10a^3+6a^5+d+3a^4d=0.\,.
$$
We eliminate $d$ to see:
$$d=-(3+3a^2+2a^4)(2a)^{-1}=-(10a^3+6a^5)(1+3a^4)^{-1}\,.
$$
After cross multiplying and simplifying, we obtain
$3(-1+a^2)^2(1+3a^2+2a^4)=0$. This implies $a=\pm1$ which is not permitted.
Thus, as claimed earlier, Case 2 is impossible.

\medbreak\noindent{\bf Case 3.} Suppose $Se_3\ne\pm e_3$ and $\rho$ is definite.
We may assume $e=1$ and $f=0$.
$$
\rho(e_1,e_1)=\rho(e_2,e_2)=\rho(e_3,e_3)\ne0\text{ and }\rho(e_i,e_j)=0\text{ for }i\ne j\,.
$$
We find $\{x,y\}$ with $x^2+y^2=1$ with $y\ne0$ so
$$\begin{array}{ccc}
Se_1=xe_1+ye_3,&Se_2=-e_2,&Se_3=-ye_1-xe_3,\\[0.05in]
Se^1=xe^1+ye^3,&Se^2=-e^2,&Se^3=ye^1-xe^3.
\end{array}
$$
We compute
$$\begin{array}{lll}
0=\Delta_{12}{}^3=cy^2,&0=\Delta_{22}{}^1=1-x+ay,&0=\Delta_{12}{}^2=1-x+by.
\end{array}$$
We obtain $c=0$, $a=b$, and $x=1+by$. We impose these relations and compute:
\begin{eqnarray*}
&&0=\Delta_{23}{}^2=-2a-y-a^2y\text{ so }\\
&&x=(1-a^2)(1+a^2)^{-1}\text{ and }y=-2a(1+a^2)^{-1}\,.
\end{eqnarray*}
We note $x^2+y^2=1$. Furthermore, since $y\ne0$, $a\ne0$.
The relations $0=\Delta_{11}{}^1$ and $0=\Delta_{13}{}^1$ imply
$$
0=3-3a^2+2a^4+2ad\text{ and }0=1-3a^2+4a^4+ad-a^3d\,.$$
If $a=\pm1$, the second equation is inconsistent and thus $a\ne\pm1$ so $x\ne0$.
We set $d=-\frac{(3-3a^2+2a^4)}{2a}$ and substitute this into the second relation
to see $2a-4a^3=0$. Since $a\ne0$, $a=\pm\frac1{\sqrt2}$.
We solve to see $b=\pm\frac1{\sqrt2}$ and $d=\mp\sqrt2$.
This gives rise to two possibilities:
$$\begin{array}{lllllllll}
a=b=\frac1{\sqrt2},&c=0,&d=-\sqrt2,&e=1,&f=0,
&x=\frac13,&y=-\frac{2\sqrt2}3,\\ [0.05in]
a=b=-\frac1{\sqrt2},&c=0,&d=\sqrt2,&e=1,&f=0,
&x=\frac13,&y=\frac{2\sqrt2}3.
\end{array}$$
We remark that $x=\frac13$ corresponds to $\tilde x=\frac13$ in the analysis of Case 1 in the proof of
Lemma~\ref{L7.3}; this is, of course, not an accident that this value surfaces again.
If we consider $e_2\rightarrow-e_2$ and $e_3\rightarrow-e_3$, we simply interchange
these two solutions. So there is really only one solution. This gives rise to the exceptional case given
in Assertion~(3b).

Let $\tilde S\in G_\Gamma^+$ have order $2$. Then there exists $\xi:=xe_1+ye_2$  for some $x^2+y^2=1$
so $\tilde S\xi=-\xi$. But then $\rho(\nabla_{\xi}\xi,\xi)=0$. Expanding this out yields the relation
$-3xy^2+x^3=0$. Since $x^2+y^2=1$, either $x=0$ and $\xi=\pm e_2$ or
 $x=\pm\frac{\sqrt3}2$ and $y=\pm\frac12$.
Thus the line thru $\xi$ is a rotation of $\pm\frac{2\pi}3$ from the line through $e_2$ and the argument
above shows $\tilde S$ is unique. This shows that $\{S,TST^{-1},T^2ST^{-2}\}$ are the 3
elements of order 2 and conjugation by $T$ permutes them cyclically. Let $\tilde T$ be another element
of order $3$. By Assertion~2 of Lemma~\ref{L7.1}, if $\tilde T$ is another element of order $3$, then
either $\tilde TT$ has order 2 or $\tilde TT^2=\id$ or $\tilde TT^2$ has order 2. But in any event,
this implies $\tilde T$ belongs to the subgroup generated by $T$ and $S$.

\subsection{The proof Theorem~\ref{T1.12}~(4)}\label{S7.7}
 Suppose every element of $G_\Gamma^+$ is of order 2. Then
$ABAB=\id$ implies $ABA^{-1}=B$ since every element is idempotent. Thus $G_\Gamma^+$ is Abelian.
We can simultaneously diagonalize the elements of $G_\Gamma^+$. Since we are dealing with $\mathbb{R}^3$,
either $G_\Gamma^+=\mathbb{Z}_2\oplus\mathbb{Z}_2$ or $G_\Gamma^+=\mathbb{Z}_2$.
We suppose the former possibility pertains.
 Let $S_1$ and $S_2$ generate $\mathbb{Z}_2\oplus\mathbb{Z}_2$. These
idempotent matrices commute and thus can be simultaneously diagonalized. Since
$\det(S_i)=+1$, each $S_i$ has two $-1$ eigenvalues and one $+1$ eigenvalue and
$H=\{\id,T_1,T_2,T_3\}$. If $\Gamma$ is invariant under this action, then $\Gamma_{ij}{}^k$ must
contain each index exactly once. We compute:
$$
\Gamma_{12}{}^3=a^3,\quad\Gamma_{13}{}^2=a^3,\quad\Gamma_{23}{}^1=a^1,\quad
\rho=-2\operatorname{diag}(a^2a^3,a^1a^3,a^1a^2)\,.
$$
If we rescale and let $\tilde e_i=\mu_ie_i$, then
$$
(a^1,a^2,a^3)
\rightarrow(\mu_2\mu_3\mu_1^{-1}a^1,\mu_1\mu_3\mu_2^{-1}a^2,\mu_1\mu_2\mu_3^{-1}a^3)\,.
$$
We can certainly rescale to assume $a^1=1$. To preserve this normalization, we require
$\mu_1=\mu_2\mu_3$. We then have $\tilde a^2=\mu_2\mu_3\mu_3\mu_2^{-1}a^2=\mu_3^2a^{-2}$
and $\tilde a^3=\mu_2\mu_3\mu_2\mu_3^{-1}=\mu_2^2a^3$. Consequently we can obtain
$a^2=\pm1$ and $a^3=\pm1$. So the possibilities become
$\vec a\in\{(1,-1,-1),(1,-1,1),(1,1,-1),(1,1,1)\}$. By permuting the elements we can get
$\vec a\in\{(1,-1,-1),(1,1,-1),(1,1,1)\}$. By replacing $e_i\rightarrow -e_i$, we can get
$\vec a\in\{(1,-1,-1),(1,1,1)\}$ as claimed.
The first possibility is discussed in Case~4a.
The second case contains the permutation $e_1\rightarrow e_2\rightarrow e_3\rightarrow e_1$
and is discussed in Assertion~3b; the structure group is $a_4$ and is not $\mathbb{Z}_2\oplus\mathbb{Z}_2$.
\hfill\qed

\section*{Acknowledgments} It is a pleasure to acknowledge useful conversations on this subject with
Professors T. Arias-Marco, M. Brozos-V\'{a}zquez, E. Garc\'{i}a-R\'{i}o, O. Kowalski, A. Kwok, and E. Puffini.
Research partially supported by project GRC2013-045 (Spain) and by the Basic Science
Research Program through the National Research Foundation of
Korea (NRF) funded by the Ministry of Education (2014053413).

\end{document}